\documentclass{amsart}

\usepackage{graphicx,amssymb,amsmath,xcolor,enumerate}
\usepackage{hyperref}
\hypersetup{
	bookmarks=true,         
	pdffitwindow=false,     
	pdfstartview={FitH},    
	colorlinks=true,      
	citecolor=red,
}

\usepackage{enumitem}

\usepackage{float}

\makeatletter
\def\l@subsection{\@tocline{2}{0pt}{2.5pc}{5pc}{}}
\makeatother

\usepackage{geometry}
\DeclareMathOperator{\su}{supp}
\DeclareSymbolFont{largesymbol}{OMX}{yhex}{m}{n}
\DeclareMathAccent{\Widehat}{\mathord}{largesymbol}{"62}
\newcommand*\di{\mathop{}\!\mathrm{d}}

\def\e{\epsilon}

\numberwithin{equation}{section}              
\newtheorem{theorem}{Theorem}[section]

\newtheorem{lemma}{Lemma}[section]
\newtheorem{proposition}{Proposition}[section]
\newtheorem*{proposition*}{Proposition}

\newtheorem*{corollary*}{Corollary}
\newtheorem{definition}{Definition}[section]
\newtheorem*{definitions*}{Definitions}

\newtheorem*{acknowledgements*}{Acknowledgements}

\newtheorem*{conjecture*}{\bf Conjecture}
\newtheorem{example}{\bf Example}[section]
\newtheorem*{example*}{\bf Example}
\theoremstyle{remark}
\newtheorem{remark}{\bf Remark}[section]

\begin{document}
\date{}                                     

\author{Yu Gao}
\address[Y. Gao]{Department of Applied Mathematics, The Hong Kong Polytechnic University, Hung Hom, Kowloon, Hong Kong}
\email{mathyu.gao@polyu.edu.hk}

\author{Hao Liu}
\address[H. Liu]{Department of Mathematics, The University of Hong Kong, Pokfulam, Hong Kong}
\email{u3005509@connect.hku.hk}

\author{Tak Kwong Wong}
\address[T. K. Wong]{Department of Mathematics, The University of Hong Kong, Pokfulam, Hong Kong.}
\email{takkwong@maths.hku.hk}

\title[Conservative solns to Hunter-Saxton eqt]{Regularity Structure  of conservative solutions to the Hunter-Saxton equation}

\maketitle

\begin{abstract}
In this paper we characterize the regularity structure, as well as show the global-in-time existence and uniqueness, of (energy) conservative solutions to the Hunter-Saxton equation by using the method of characteristics. The major difference between the current work and previous results is that we are able to characterize the singularities of energy measure and their nature in a very precise manner. In particular, we show that singularities, whose temporal and spatial locations are also explicitly given in this work, may only appear at at most countably many times, and are completely determined by the absolutely continuous part of initial energy measure. Our mathematical analysis is based on using the method of characteristics in a generalized framework that consists of the evolutions of solution to the Hunter-Saxton equation and the energy measure. This method also provides a clear description of the semi-group property for the solution and energy measure for all times.

\end{abstract}
{\small
\keywords{\textbf{\textit{Keywrods}:} 
formulation of singularity,   well-posedness, integrable system, decomposition of energy measure, semi-group property}
}
{
\hypersetup{linkcolor=blue}
\tableofcontents
}

\section{Introduction}
In this paper, we study the regularity structure of (energy) conservative solutions to the Hunter-Saxton equation on the whole real line $\mathbb{R}$:  for $x$, $t\in\mathbb{R}$,
\begin{align}\label{eq:HS}
u_t+uu_x=\frac{1}{2} \int_{-\infty}^x u_y^2 (y, t)\di y.
\end{align}
Here, $u_t$ and $u_x$ represent the time and spatial derivative of $u$, respectively. The equation~\eqref{eq:HS} is an integrable equation that was first proposed by Hunter and Saxton \cite{hunter1991dynamics} to study a nonlinear instability in the director field of a nematic liquid crystal.
When $u$ is a classical solution to \eqref{eq:HS}, differentiating \eqref{eq:HS} with respect to the spatial variable $x$ yields
\begin{align}\label{eq:HSD}
(u_{t}+uu_{x})_x=\frac{1}{2}u_x^2,
\end{align}
and hence, the following energy conservation law holds:
\begin{align}\label{eq:conservationlaw}
(u_x^2)_t+(uu_x^2)_x=0.
\end{align}
According to equation \eqref{eq:conservationlaw}, it is nature to find the global-in-time conservative solution $u$ that satisfies 
$\|u_x(\cdot,t)\|_{L^2}=\|\bar{u}_x\|_{L^2}$ for any $t\in\mathbb{R}$, provided that the initial datum $u|_{t=0}=\bar{u}$ satisfies $\bar{u}_x\in L^2(\mathbb{R})$. However, similar to the inviscid Burgers equation, using the characteristics method, one can verify the following fact: if the initial datum {$\bar{u}$ satisfies $\inf_{x\in\mathbb{R}}\bar{u}_x(x)<0$,} then we have $u_x(\cdot,t)\to-\infty$ as $t\to T_0:=-2/\inf_{x\in\mathbb{R}}{\bar{u}_x(x)}.$ At the blow-up time, the function $u$ will lose the information of energy conservation temporarily. This can be seen from the following simple example.
Consider the solution with initial datum (see also in \cite{hunter1995nonlinear,bressan2007asymptotic})
\begin{equation*}
\bar{u}(x)=\left\{
\begin{aligned}
&0,\quad x\leq0,\\
&-x,\quad 0<x< 1,\\
&-1,\quad x\geq 1.
\end{aligned}
\right.
\end{equation*}
The conservative solution $u$ is explicitly given by
\begin{equation*}
u(x,t)=\left\{
\begin{aligned}
&0,\quad x\leq0,\\
&-\frac{2x}{2-t},\quad 0<x\leq \left(1-\frac{t}{2}\right)^2,\\
&-1+\frac{t}{2},\quad x\geq  \left(1-\frac{t}{2}\right)^2,
\end{aligned}
\right.\quad
u_{x}(x,t)=\left\{
\begin{aligned}
&-\frac{2}{2-t},\quad 0<x\leq \left(1-\frac{t}{2}\right)^2,\\
&0,\quad ~\textrm{otherwise.}
\end{aligned}
\right.
\end{equation*}
We have $u(x,2)\equiv 0$ and $u_x^2(\cdot,t)\to \delta$ in the sense of distributions as $t\to 2$, where $\delta$ is the Dirac delta mass at the origin $x=0$. Notice that $\|u_x(\cdot,t)\|_{L^2}= 1$ for any $t\neq 2$ and $\|u_x(\cdot,2)\|_{L^2}=0$. Therefore, if we only study the solution $u$, the energy $\|u_x(\cdot,t)\|_{L^2}$ is not conserved {only} at $t=2$; this viewpoint, which causes a temporary/momentary disappearance of energy, is very non-physical. {It is worth noting that at the blow-up time $t=2$, the energy density is converted into the singular measure $\delta$, and $u_x^2(\cdot,2)\equiv 0$ corresponds to the absolutely continuous part of the measure $\delta$ with respect to the Lebesgue measure $\mathcal{L}$.} 
To describe the energy conservation of weak solutions, we use a nonnegative Radon measure $\mu(t)$, which can be seen as the energy measure, to replace $u_x^2(\cdot,t)$ in equations \eqref{eq:HS} and \eqref{eq:conservationlaw}, and study the Hunter-Saxton equation in the following  generalized framework:

\

\noindent\underline{\textbf{Generalized Framework}}:
\begin{align}
&u_t+uu_x=\frac{1}{2}\int_{-\infty}^x\di \mu(t),\label{eq:gHS1}\\
&\mu_t+(u\mu)_x=0,\label{eq:gHS2}\\
&\di \mu_{ac}(t)=u_x^2(\cdot,t)\di x.\label{eq:gHS3}
\end{align}

\

\noindent Let us comment on the notations in \eqref{eq:gHS1}-\eqref{eq:gHS3} as follows. In \eqref{eq:gHS3}, the measure $\mu_{ac}(t)$ stands for the absolutely continuous part of  the energy measure $\mu(t)$  with respect to the Lebesgue measure. 
{
	In \eqref{eq:gHS1}, the integral $\int_{-\infty}^x\di \mu(t)$ is not well defined when $\mu(t)$ contains pure point measures. However, we are going to show that there are only at most countably many time $t$ such that  the energy measure $\mu(t)$ has pure point measures (see Theorem \ref{thm:measure} (iii) for details), and hence, the ambiguity of this integral will not affect the definition of weak (energy conservative) solutions below. Therefore, we will keep this integral notation in \eqref{eq:gHS1}.
Equation \eqref{eq:gHS2} is viewed in the sense of distributions.}
We then study the  conservative solutions to the generalized framework \eqref{eq:gHS1}-\eqref{eq:gHS3}.
First of all, let us define the following space for conservative solutions:
\begin{definition}
Let $\mathcal{D}$ be the set of pairs $(u,\mu)$ satisfying
\begin{enumerate}
\item[(i)] $u\in C_b(\mathbb{R})$, $u_x\in L^2(\mathbb{R})$;
\item[(ii)] $\mu\in \mathcal{M}_+(\mathbb{R})$;
\item[(iii)] $\di \mu_{ac}=u_x^2\di x$, where  $\mu_{ac}$ is the absolutely continuous part of measure $\mu$ with respect to the Lebesgue measure $\mathcal{L} $.
\end{enumerate}
Here, $\mathcal{M}_+(\mathbb{R})$ stands for the set of finite positive Radon measure endowed with the weak topology. 
\end{definition}
Following the physical meaning of the energy measure $\mu(t)$, we require  $\mu(t) \in \mathcal{M}_+(\mathbb{R})$, so that the total energy is finite; however, technically speaking, our mathematical analysis in this paper only requires that the positive measure $\mu(t)((-\infty, x]) < \infty$ for all $x \in \mathbb{R}$.
There are several different definitions for conservative solutions to the Hunter-Saxton  equation; see \cite[Definition 4.4]{Bressan2010} for instance.
In this paper, we will use the following definition of conservative solutions  to the Hunter-Saxton  equation:
\begin{definition}[Conservative solutions]\label{def:weak}
Let $(\bar{u},\bar{\mu})\in \mathcal{D}$ be a given initial datum. The pair $(u(t), \mu(t))$ is said to be a global-in-time conservative solution to the generalized framework \eqref{eq:gHS1}-\eqref{eq:gHS3} with the initial datum $(\bar{u},\bar{\mu})$, if the pair $(u(t),\mu(t))$ satisfies the following:
\begin{enumerate}
\item[(i)] $u\in C(\mathbb{R};C_b(\mathbb{R}))\cap C^{1/2}_{loc}(\mathbb{R}\times\mathbb{R})$, $u_t\in L_{loc}^2(\mathbb{R}\times\mathbb{R})$, $u_x(\cdot, t) \in L^2(\mathbb{R})$ for all $t \in \mathbb{R}$,  and $\mu\in C(\mathbb{R};\mathcal{M}_+(\mathbb{R}))$;
\item[(ii)] $(u(\cdot,0),\mu(0))=(\bar{u},\bar{\mu})$, and $\di\mu(t)={u}_x^2(x,t)\di x$ for a.e. $t\in\mathbb{R}$;
\item[(iii)] the equation
\begin{align}\label{eq:weakformula}
\int_{\mathbb{R}}\int_{\mathbb{R}}u\phi_t-\phi\left(uu_x-\frac{1}{2}F\right)\di x\di t=0
\end{align}
holds for all $\phi\in C_c^\infty(\mathbb{R}\times \mathbb{R})$; the function $F(x,t)$ is defined by  $F(x,t) :=\int_{-\infty}^x\di \mu(t)$; 
\item[(iv)] conservation of energy:
\begin{equation}\label{eq:fourth}
\int_{\mathbb{R}}\int_{\mathbb{R}} \left( \phi_{t} + u \phi_{x}\right) \di \mu(t) \di t =0
\end{equation}
for every $\phi\in C^{\infty}_c(\mathbb{R}\times \mathbb{R})$;
\item[(v)] equation \eqref{eq:gHS3} holds for all $t \in \mathbb{R}$. 
\end{enumerate}
The last condition is equivalent to say that $(u(t), \mu(t)) \in \mathcal{D}$ for all $t \in \mathbb{R}$.   
\end{definition}
It follows from the regularity requirement (i) in Definition \ref{def:weak} that the condition (iii) in Definition \ref{def:weak} (namely Identity \eqref{eq:weakformula} holds for all test functions) is equivalent to the following condition:
\begin{enumerate}
\item [(iii)'] Equation \eqref{eq:gHS1} holds in the $L^2_{loc}(\mathbb{R}\times\mathbb{R})$ sense.
\end{enumerate}
The condition (iii)' is equivalent to \cite[Definition 1.1(ii)]{bressan2007asymptotic}, when the flux is given by $u^2/2$.

\begin{remark}[Energy conservation]\label{rem: mu(t)(mathbb R) is conserved}
	Indeed, any weak solution $(u, \mu)$ to the Hunter-Saxton equation under the generalized framework \eqref{eq:gHS1}-\eqref{eq:gHS3} in the sense of Definition \ref{def:weak} satisfies
	the energy conservation $\mu(t)(\mathbb{R}) = \bar{\mu}(\mathbb{R})$ for $t\in \mathbb{R}$. This is why we call it a conservative solution. To verify this energy conservation, we now consider any arbitrary time $t>0$ (the case for $t<0$ will be similar).
	We choose non-negative smooth  functions $\chi_\e(s)$ and $g_{R}(x)$ for $\e>0$ and $R>0$, where 	$\chi_\e(s)=0$ for $s \le -\e$ and $s \ge t+\e$. 
	$\chi_\e(s)=1$ for $s\in (0, t)$,  $\chi_\e'(s) \ge 0$ for $s \in (-\e,0)$ and $\chi_\e'(s) \le 0$ for $s \in (t, t+\e)$.  The function $g_R$ satisfies $g_{R}(x) = 1$ for $|x| \le R$, $g_{R}(x) = 0$ for $|x| \ge 2R$ and $\left|g'_{R}(x)\right| \le \frac{2}{R}$. Finally, let  $\phi(x,t) = \chi_\e(t) g_{R}(x)$, and substituting this $\phi$ into equation \eqref{eq:fourth}, we obtain  
	\begin{equation*}
	\int_{\mathbb{R}} \int_{\mathbb{R}} \chi'_\e (s) g_{R}(x) \di \mu(s)\di s + \int_{\mathbb{R}} \int_{\mathbb{R}} u(x, s) \chi_\e (s) g'_{R}(x) \di \mu(s)  \di s =0.
	\end{equation*}
	Passing to the limit as $R \to \infty$, and using $u$ is bounded, we have
	\begin{equation*}
	\int_{-\e}^{0}\chi'_\e (s) \mu(s)(\mathbb{R})  \di s + \int_{t}^{t+\e}\chi'_\e (s) \mu(s)(\mathbb{R})  \di s =0.
	\end{equation*}
	Since $\mu \in C(\mathbb{R};\mathcal{M}_+(\mathbb{R}))$, 
	\begin{equation*}
	\begin{aligned}
	\left|\int_{-\e}^{0}\chi'_\e (s) \mu(s)(\mathbb{R})  \di s -  \bar{\mu}(\mathbb{R})\right|
	=&\left|\int_{-\e}^{0}\chi'_\e (s) \mu(s)(\mathbb{R})  \di s - \int_{-\e}^{0}\chi'_\e (s) \bar{\mu}(\mathbb{R})  \di s \right|\\
	\le& \int_{-\e}^{0}\chi'_\e (s) \left|\mu(s)(\mathbb{R}) - \bar{\mu}(\mathbb{R})    \right| \di s \to 0 ~\textrm{as} ~\e \to 0^+.
	\end{aligned}
	\end{equation*}
	This shows that $\int_{-\e}^{0}\chi'_\e (s) \mu(s)(\mathbb{R}) \di s \to  \bar{\mu}(\mathbb{R}) $ as $\e \to 0^+$.
	On the other hand, one can also verify that $\int_{t}^{t+\e}\chi'_\e (s) \mu(s)(\mathbb{R})  \di s $ $\to $ $- \mu(t)(\mathbb{R})$ as $\e \to 0^+$ in a similar manner, and hence, we finally obtain $\mu(t)(\mathbb{R}) = \bar{\mu}(\mathbb{R})$. 
	
\end{remark}

\begin{remark}\label{eq: u_x in Linfity L2}
	It follows from Remark \ref{rem: mu(t)(mathbb R) is conserved} and Equation \eqref{eq:gHS3} that we also have $u_x \in L^{\infty}(\mathbb{R}; L^2(\mathbb{R}))$.
\end{remark}

\begin{remark}
	It follows $\mu\in C(\mathbb{R};\mathcal{M}_+(\mathbb{R}))$ and (ii) in the Definition \ref{def:weak} that $F(\cdot ,t)$ is well defined for a.e. $t \in \mathbb{R}$, and $F \in L^{\infty}_{loc}(\mathbb{R}\times \mathbb{R})$. Moreover, according to Remark \ref{eq: u_x in Linfity L2}, we also know   $F \in L^\infty( \mathbb{R}\times\mathbb{R})$.
\end{remark}
Under the above definition, we have the following results:
\begin{theorem}\label{thm:introduce}
Let $(\bar{u},\bar{\mu})\in\mathcal{D}$ be given. Then there exists a unique global-in-time  conservative solution (in the sense of Definition \ref{def:weak}) $(u(t),\mu(t))$ to the generalized framework \eqref{eq:gHS1}-\eqref{eq:gHS3} with  $(\bar{u},\bar{\mu})$ as its initial datum.  Let $\mu(t)=\mu_{ac}(t)+\mu_{pp}(t)+\mu_{sc}(t)$, where $\mu_{ac}(t)$, $\mu_{pp}(t)$, and $\mu_{sc}(t)$ are the absolutely continuous part, pure point part and the singular continuous part of $\mu$ respectively.
 The following statements also hold:
\begin{enumerate}
\item [(i)] (Energy conservation) we have $\mu \in C(\mathbb{R}; \mathcal{M}_+(\mathbb{R}))$ and
\begin{align}\label{eq:conserE}
\mu(t)(\mathbb{R})= \bar{\mu}(\mathbb{R}),\quad \mbox{for all }~t\in\mathbb{R}.
\end{align}
\item [(ii)] (Formation of singularity, and its temporal and spatial locations) For any $t\neq0$, $\mu_{pp}(t)$ and $\mu_{sc}(t)$ are determined by the absolutely continuous part $\bar{\mu}_{ac}$, namely determined by $\bar{u}_x$. More precisely, for any $t\neq 0$, we define
\begin{align}\label{eq:singulartE}
A_t^E =\left\{x:~~\bar{u}_x(x)=-\frac{2}{t}\right\}.
\end{align}
If $\mathcal{L}(A_t^E)\neq 0$, then $\mu_{pp}(t)+\mu_{sc}(t)\neq 0$  (i.e., $\mu_{pp}(t)+\mu_{sc}(t)$ is not a zero measure) and
\begin{align}\label{eq:support}
\mathrm{supp}\Big(\mu_{pp}(t)+\mu_{sc}(t)\Big)\subset \left\{x+\bar{u}(x)t+\frac{t^2}{4}\bar{\mu}((-\infty,x)): ~ x\in A_t^E\right\}.
\end{align}
All the intervals with positive length  in $A_t^E$ will generate the pure point part $\mu_{pp}(t)$, and the rest of $A_t^E$ will generate the singular continuous part $\mu_{sc}(t)$.

\item [(iii)] (Countably many singular times) There are at most countably many times $t\in \mathbb{R}$, such that either the pure point part or the singular continuous part of $\mu(t)$ is not identically equal to zero, namely both the sets
\begin{align}\label{eq:singulartimeE}
T_p:=\{t:~\mu_{pp}(t)\neq 0\}\quad\mbox{and}\quad	T_s:=\{t:~\mu_{sc}(t)\neq 0\}
\end{align}
are countable. 
\item [(iv)] (Regularity) For all time $t\in\mathbb{R}$, the function $u(x,t)$ is globally absolutely continuous with respect to the spatial variable $x$ and  
\begin{align}\label{eq:acpartE}
\di\mu_{ac}(t)=u_x^2(x,t)\di x.
\end{align}
Furthermore,
\begin{align}\label{eq:propertiesuE}
u\in C(\mathbb{R};C_b(\mathbb{R}))\cap C^{1/2}_{loc}(\mathbb{R}\times\mathbb{R}),~~ u_x\in L^\infty(\mathbb{R};L^2(\mathbb{R})),~~\mbox{and}~~ u_t\in L_{loc}^2(\mathbb{R}\times \mathbb{R}).
\end{align}
\item [(v)] (Asymptotic behaviour) 
If $\bar{u}(-\infty):=\lim_{x\to-\infty}\bar{u}(x)$ exists, then we  have
\begin{align}\label{eq:leftfarE}
\lim_{x\to-\infty}u(x,t)= \bar{u}(-\infty).
\end{align}
On the other hand, if $\bar{u}(+\infty):=\lim_{x\to +\infty}\bar{u}(x)$ exists, then we also have
\begin{align}\label{eq:rightfarE}
\lim_{x\to+\infty}u(x,t)= \bar{u}(+\infty)+\frac{1}{2}\bar{\mu}(\mathbb{R})t.
\end{align}

\end{enumerate}
\end{theorem}


Let us start with the absolutely continuous initial data (as described below) to illustrate our main strategies to prove the above results. Consider an initial datum $(\bar{u},\bar{\mu})\in \mathcal{D}$. When the initial energy measure $\bar{\mu}$ is absolutely continuous with respect to Lebesgue measure $\mathcal{L}$ (i.e., $\bar{\mu}\ll \mathcal{L}$ and $\di \bar{\mu}=\bar{u}_x^2\di x$), we can apply the usual Lagrangian coordinates $(\xi,t)$ to define the flow map via $\partial_tX(\xi,t)=u(X(\xi,t),t)$ and $X(\xi,0)=\xi$, and obtain the following global characteristics (see also \eqref{eq:X} below):
\begin{align}\label{eq:flowmap1}
X(\xi,t)=\xi+\bar{u}(\xi)t+\frac{t^2}{4} \int_{(-\infty,\xi)}\bar{u}_x^2(y)\di y=\xi+\bar{u}(\xi)t+\frac{t^2}{4}\bar{\mu}((-\infty,\xi)).
\end{align}
Since $X_\xi(\xi,t)=[1+\frac{t}{2}\bar{u}_x(\xi)]^2\geq 0$, $X(\cdot,t)$ is an increasing function for any fixed $t$. The solution $(u,\mu)$ can be recovered by $u(X(\xi,t),t)=\partial_tX(\xi,t)$ and $ \mu(t)=X(\cdot,t)\#(\bar{u}_x^2 \di x) $.  Then, we can apply an {elementary} result for push-forward measures (see Lemma \ref{lmm:keylemma} for instance) to {decompose} the measure $\mu(t)$ into its absolutely continuous part, pure point part, and singular continuous part by using the derivative $X_\xi(\cdot,t)$. More precisely, 
the singular part of $\mu(t)$ is determined by $\{\xi:~~X_\xi(\xi,t)=0\}$, or equivalently $\{x:~~\bar{u}_x(x)=-\frac{2}{t}\}$; in particular, 
the intervals in $\{x:~~\bar{u}_x(x)=-\frac{2}{t}\}$ generate the pure point part of $\mu(t)$, and the rest points in  $\{x:~~\bar{u}_x(x)=-\frac{2}{t}\}$ corresponds to the singular continuous part of $\mu(t)$.
The absolutely continuous part of $\mu(t)$ is given by $\{\xi:~~X_\xi(\xi,t)>0\}$. 

The Lagrangian coordinates $(\xi,t)$ are applicable to all absolutely continuous initial data (i.e. the initial energy measure $\bar{\mu}$ is absolutely continuous with respect to Lebesgue measure). For such an initial datum $(\bar{u},\bar{\mu})$, since there is no point mass at any particular point $\xi$, the cumulative energy distribution in \eqref{eq:flowmap1} satisfies 
\begin{align}\label{eq:ced}
\bar{\mu}((-\infty,\xi))=\bar{\mu}((-\infty,\xi])=\int_{(-\infty,\xi)}\bar{u}_x^2(y)\di y.
\end{align}
However, if the initial energy measure $\bar{\mu}$ contains any pure point part, then the above relation is not true and it is impossible to obtain global flow map $X$ in the Lagrangian coordinates $(\xi,t)$.
To overcome the difficulty brought by the singular parts of $\bar{\mu}$, we will apply another set of variables used in \cite{Bressan2010,antonio2019lipschitz,bressan35uniqueness}. We ``flatten" the singular part of $\bar{\mu}$ with the help of  an increasing Lipschitz function $\bar{x}:\mathbb{R} \to \mathbb{R}$ defined by
\[
\bar{x}(\alpha)+\bar{\mu}((-\infty, \bar{x}(\alpha)))\leq \alpha \leq \bar{x}(\alpha)+\bar{\mu}((-\infty, \bar{x}(\alpha)]),\quad \alpha\in\mathbb{R}.
\]
Then, the function $f(\alpha):=1-\bar{x}'(\alpha)\geq0$ will play the role of the energy density in the $\alpha$-variable; see Remark \ref{rmk:cumulative} (i) for the explanation. We actually have $\bar{\mu} = \bar{x}\#(f\di\alpha)$ in this case; see Proposition \ref{pro:keypro} for further details. Since $\bar{x}$ is increasing and $f\in L^1(\mathbb{R})$, the absolutely continuous part of the push-forward measure $\bar{\mu}=\bar{x}\# (f\di\alpha)$ corresponds to $B_0^L=\{\alpha:~~\bar{x}'(\alpha)>0\}$, and the singular part of $\bar{\mu}$ is determined by $A_0^L=\{\alpha:~~\bar{x}'(\alpha)=0\}$; see Lemma \ref{lmm:keylemma} for instance. 
Moreover, similar to $X(\xi, t)$, we can also define the global characteristics $y(\alpha,t)$ in the $\alpha$-variable depending only on the initial datum with an explicit form \eqref{eq:xbarbeta}, and $y(\alpha,t)$ is increasing with respect to $\alpha$ at any particular time $t\in\mathbb{R}$. Then, the global conservative solution $(u,\mu)$ to the generalized framework \eqref{eq:gHS1}-\eqref{eq:gHS3} will be recovered via $u(y(\alpha,t),t)=\partial_ty(\alpha,t)$ and $\mu(t)=y(\cdot,t)\#(f\di\alpha)$. We can also obtain the following important property (see Proposition \ref{pro:keypro2}):
\begin{equation*}
y_{\alpha}(\alpha,t)=\left\{
\begin{aligned}
&\bar{x}'(\alpha)\left[1+\frac{t}{2}\bar{u}_x(\bar{x}(\alpha))\right]^2,\quad \alpha\in B_0^L,\\
&\frac{t^2}{4},\quad \alpha \in A_0^L.
\end{aligned}
\right.
\end{equation*}
The singular part of $\mu(t)$ comes from the set $\{\alpha:~~y_{\alpha}(\alpha,t)=0\}$ in the $\alpha$-coordinate, or equivalently, the set $\{x:~~\bar{u}_x(x)=-\frac{2}{t}\}$ in the Eulerian coordinates. Since $A_0^L$ corresponds to the singular part of the initial datum $\bar{\mu}$, the above formula implies that it will never create the 
 singular part of $\mu(t)$  for any $t\neq0$.
All the properties in Theorem \ref{thm:introduce} and the structure of conservative solutions are obtained with the help of the characteristics $y(\alpha,t)$.
  In a nutshell, the existence and regularity structure of conservative solutions can be shown by using the method of characteristics. The uniqueness of conservative solutions follows from the existence and uniqueness of characteristics for conservative solutions in the sense of Definition~\ref{def:weak} (see Theorem \ref{thm:main1}).
 We also note that the singular continuous part of $\mu(t)$ is inevitable even $\bar{u}$ is a smooth function; see Example \ref{ex:singularremark} for instance.

For the flow map $X(\xi,t)$, we obviously have the semi-group property $X(\cdot,t)=X(\cdot,t-s)\circ X(\cdot,s)$. A similar relation between the above two types of characteristics $X(\xi,t)$ and $y(\alpha,t)$ still holds. Indeed,  
 we have the following theorem:
\begin{theorem}[Semi-group property]\label{thm:2}
Let $(u,\mu)$ be a solution to the generalized framework \eqref{eq:gHS1}-\eqref{eq:gHS3} given by Theorem \ref{thm:introduce} with initial datum $(\bar{u},\bar{\mu})\in\mathcal{D}$.
Consider a time $s\in\mathbb{R}$ such that $\mu(s)$ is absolutely continuous with respect to the Lebesgue measure. Set $\tilde{u}(x)=u(x,s)$, and $X(\xi,t)$ is defined by \eqref{eq:flowmap1} with $\bar{u}$ replaced by $\tilde{u}$ (and $\bar{\mu}((-\infty , x) ) $ is replaced by $\int_{(-\infty, x)} \tilde{u}_x^2(x) \di x$).   Let $y(\alpha,t)$ be the characteristics (in the $\alpha$-coordinate) corresponding to $(\bar{u},\bar{\mu})$, see \eqref{eq:xbarbeta} for its definition. Then, we have
\begin{align*}
\tilde{u}\in C_b(\mathbb{R}), \quad \tilde{u}_{x}\in L^2(\mathbb{R}),\quad \| \tilde{u}_{x}\|_{L^2}=\bar{\mu}(\mathbb{R}).
\end{align*}
For any $t\in\mathbb{R}$, we also have
\begin{align*}
y(\cdot,t)=X(\cdot,t-s)\circ y(\cdot,s), \quad \mu(t)=X(\cdot,t-s)\#(\tilde{u}_x^2\di x),
\end{align*}
and
\begin{align*}
u(x,t)=\frac{\partial}{\partial t}X(\xi,t-s)=\tilde{u}(\xi) +\frac{(t-s)}{2}\int_{(-\infty,\xi)}\tilde{u}_x^2(y)\di y~\textrm{ for }~ x=X(\xi,t-s).
\end{align*}

\end{theorem}
Global characteristics similar to $y(\alpha,t)$ were also used in \cite{Bressan2010,antonio2019lipschitz} to construct the semi-group property of conservative solutions to the Hunter-Saxton equation, and obtain the Lipschitz metrics for stability. In \cite{bressan35uniqueness}, similar variables were used to show the uniqueness of conservative solutions to the Camassa-Holm equation via the characteristics method. 
There is another different set of variables  defined by 
\[
\bar{k}(\eta)=\inf\{x:~~\bar{\mu}((-\infty,x])\geq\eta\},
\]
provided that a Radon measure $\bar{\mu}\in\mathcal{M}_+(\mathbb{R})$ is given. Global characteristics $k(\eta,t)$ can also be defined globally via initial datum $(\bar{u},\bar{\mu})$; see \cite{bressan2007asymptotic,antonio2019lipschitz} for instance.
In \cite{bressan2007asymptotic}, Bressan, Zhang and Zheng used this kind of characteristics to study the following more general model:
\[
u_t+g(u)_x=\frac{1}{2}\int_0^xg''(u)u_x^2\di x.
\]
Here, the flux $g$ is a function with a Lipschitz continuous second-order derivative such that $g''(0)>0$. When $g(u)=u^2/2$, the above equation becomes  the Hunter-Saxton equation. They obtained global existence and uniqueness of conservative solutions on the half real line by using the above characteristics for compactly supported initial measure $\bar{\mu}$. This kind of characteristics was also used in \cite{antonio2019lipschitz} to study the stability of conservative solutions to the Hunter-Saxton equation.
With the above map $\bar{k}$, one has $\bar{\mu}=\bar{k}\# \mathcal{L}|_{[0,\bar{\mu}(\mathbb{R})]}.$  Comparing with $\bar{x}$ (which is globally Lipschitz) that we use in this paper, the function $\bar{k}$, however,  has potential jump discontinuities.

Except for conservative solutions, the Hunter-Saxton equation also has a class of weak solutions called  dissipative solutions which dissipate the energy. Hunter and Zheng \cite{hunter1995nonlinear1,hunter1995nonlinear} established the global existence of both dissipative and conservative weak solutions to \eqref{eq:HSD} with initial data $\bar{u}_x\in BV(0,\infty)$.  Then, Zhang and Zheng  \cite{zhang1998oscillations,zhang1999existence} studied the global existence of dissipative weak solutions to \eqref{eq:HSD} with nonnegative initial data $\bar{u}_x\in L^p(\mathbb{R})$  for any $p \geq2$. Later in \cite{zhang2000existence}, they obtained global solutions (in both the dissipative and conservative cases) to the Hunter-Saxton equation on the half-line with any initial data $\bar{u}_x\in L^2(\mathbb{R})$ by the theory of Young measures. In \cite{bressan2005global}, Bressan and Constantin rewrote the equation in terms of new variables, and obtained global existence and uniqueness of dissipative solutions. The uniqueness of dissipative solutions were also studied by the characteristics methods  in \cite{cieslak2016maximal,dafermos2011generalized,dafermos2012maximal}. 

The rest of this paper is organized as follows. In Section \ref{sec:structure}, we are going to introduce global characteristics and construct global conservative solutions. The structure of these solutions will be studied in details. In Section \ref{sec:existenceuniqueness}, we will show that the construction based on the method of characteristics provides the global solution in the sense of Definition \ref{def:weak}. The uniqueness of conservative solutions will be obtained via the characteristics method as well. In Appendix  \ref{app:real}, we will present some useful facts from real analysis, and construct an absolutely continuous initial datum by Cantor fat set; this initial datum will generate singular continuous part within a finite time.

\section{Structure of conservative  solutions via characteristics}\label{sec:structure}
In this section, we will introduce the characteristics to construct global conservative solutions and study their structure. 

\subsection{Global characteristics} 
To illustrate the idea, we first consider a smooth solution $u$ to the Hunter-Saxton equation \eqref{eq:HS}. 
For any $\beta\in\mathbb{R}$, we define $x(\beta,t)$ by
\begin{align}\label{eq:defx}
x(\beta,t)+\int_{(-\infty, x(\beta,t))}u_x^2(y,t)\di y=\beta,~~t\in\mathbb{R}. 
\end{align}
Differentiating the above identity with respect to $t$ and $\beta$ respectively, we obtain
\begin{equation}\label{eq:dtx&dbetax}
\partial_tx(\beta,t)=\frac{(uu_x^2)(x(\beta,t),t)}{1+u_x^2(x(\beta,t),t)},\qquad \partial_{\beta}x(\beta,t)=\frac{1}{1+u_x^2(x(\beta,t),t)},
\end{equation}
where we used \eqref{eq:conservationlaw} for computing $\partial_tx(\beta,t)$.
Then, for fixed $\alpha\in \mathbb{R}$, we define a function $\beta(t)$ by solving the following ordinary differential equation (ODE):
\begin{align}\label{eq:betat}
\beta'(t)=u(x(\beta(t),t),t),\quad \beta(0)=\alpha\in\mathbb{R}.
\end{align}
It follows from the chain rule, \eqref{eq:dtx&dbetax} and \eqref{eq:betat} that
\begin{align}\label{eq:translation}
\frac{\di }{\di t}x(\beta(t),t)=(\partial_tx+\partial_\beta x\cdot \beta')(\beta(t),t)=u(x(\beta(t),t),t)=\beta'(t).
\end{align}
Differentiating \eqref{eq:betat} with respect to the time $t$ again, we obtain
\begin{equation*}
\begin{aligned}
\beta''(t)=&\partial_tu(x(\beta(t),t),t)+\partial_xu(x(\beta(t),t),t)\frac{\di }{\di t}x(\beta(t),t)=(u_t+uu_x)(x(\beta(t),t),t)\\
=&\frac{1}{2}\int_{(-\infty,x(\beta(t),t))}u_x^2(y,t)\di y=\frac{1}{2}[\beta(t)-x(\beta(t),t)],
\end{aligned}
\end{equation*}
where the last equality comes from \eqref{eq:defx}.
Taking one more time derivative and using \eqref{eq:translation}, we have
\begin{align}\label{eq:thirdorder}
\beta'''(t)=0.
\end{align}
Denote $\bar{x}(\alpha)=x(\alpha,0)$, which is uniquely determined by
\begin{align*} 
\bar{x}(\alpha)+\int_{(-\infty, \bar{x}(\alpha))}u_x^2(y,0)\di y=\alpha.
\end{align*}
Then, we have the following initial data for the ODE \eqref{eq:thirdorder}: 
\begin{align}\label{eq:initial}
\beta(0)=\alpha,\quad \beta'(0)=u(\bar{x}(\alpha),0),\quad \beta''(0)=\frac{1}{2}(\alpha-\bar{x}(\alpha))=\frac{1}{2}\int_{(-\infty, \bar{x}(\alpha))}u_x^2(y,0)\di y.
\end{align}
Therefore, solving \eqref{eq:thirdorder} and \eqref{eq:initial} yields
\begin{align}\label{eq:global1}
\beta(t)=\alpha+u(\bar{x}(\alpha),0)t+\frac{t^2}{4}(\alpha-\bar{x}(\alpha)),
\end{align}
and hence, using \eqref{eq:translation}, we also have 
\begin{align}\label{eq:global2}
x(\beta(t),t)=\bar{x}(\alpha)+u(\bar{x}(\alpha),0)t+\frac{t^2}{4}(\alpha-\bar{x}(\alpha)).
\end{align}
Therefore, to obtain global formulas for $\beta(t)$ and $x(\beta(t),t)$, we only need the information of initial datum $u(x,0)$. 
The above idea can be generalized to non-smooth\footnote{Here, a non-smooth pair $(\bar{u},\bar{\mu})\in\mathcal{D}$ means that the initial data $\bar{\mu}$ is not absolutely continuous with respect to the Lebesgue measure $\mathcal{L}$.} initial data in $\mathcal{D}$ as described below.

Consider an initial datum $(\bar{u},\bar{\mu})\in\mathcal{D}$. For any $\alpha\in\mathbb{R}$, we can also define the function $\bar{x}$ via
\begin{align}\label{eq:barx1}
\bar{x}(\alpha)+\bar{\mu}((-\infty, \bar{x}(\alpha)))\leq \alpha \leq \bar{x}(\alpha)+\bar{\mu}((-\infty, \bar{x}(\alpha)]). 
\end{align}
If follows directly from the above definition of $\bar{x}$ that $\bar{x}(\alpha)\le \alpha$ for all $\alpha\in\mathbb{R}$.
As an analogy to \eqref{eq:global1} and  \eqref{eq:global2}, we also have global-in-time $\beta(t)$ and $x(\beta(t),t)$ for non-smooth initial datum in $\mathcal{D}$ as follows: 
\begin{align}\label{eq:global3}
\beta(t)=\alpha+\bar{u}(\bar{x}(\alpha))t+\frac{t^2}{4}(\alpha-\bar{x}(\alpha)),
\end{align}
and
\begin{align}\label{eq:global4}
x(\beta(t),t)=\bar{x}(\alpha)+\bar{u}(\bar{x}(\alpha))t+\frac{t^2}{4}(\alpha-\bar{x}(\alpha)).
\end{align}
Hence, we can think of $x(\beta(t),t)$ as a function of $\alpha$ and $t$, then we define $y(\alpha, t)$ as follows:
\begin{align}\label{eq:xbarbeta}
y(\alpha,t)=\bar{x}(\alpha)+\bar{u}(\bar{x}(\alpha))t+\frac{t^2}{4}(\alpha-\bar{x}(\alpha)).
\end{align}

Now, let us first show some properties for initial data in $\mathcal{D}$ as follows.
\begin{proposition}\label{pro:keypro}
Let $(\bar{u},\bar{\mu})\in\mathcal{D}$ and  $\bar{\mu}=\bar{\mu}_{ac}+\bar{\mu}_{pp}+\bar{\mu}_{sc},$
where $\bar{\mu}_{ac}$, $\bar{\mu}_{pp}$, and $\bar{\mu}_{sc}$ are the absolutely continuous part, pure point part and the singular continuous part of $\bar{\mu}$ respectively.  Define a map $\bar{x}:\mathbb{R}\to\mathbb{R}$ by \eqref{eq:barx1}.
Then the following statements hold:
\begin{enumerate}
\item [(i)] The function $\bar{x}$ is Lipschitz continuous  with Lipschitz constant bounded by $1$. 
\item [(ii)] Definition \eqref{eq:barx1} is the same as 
\begin{align}\label{eq:barx2}
	\bar{x}(\alpha)=\sup\left\{x:~x+\bar{\mu}((-\infty,x))\leq\alpha \right\},
\end{align}
and
\begin{align}\label{eq:barx3}
	\bar{x}(\alpha)=\inf\left\{x:~x+\bar{\mu}((-\infty,x])\geq \alpha \right\}.
\end{align}
\item [(iii)] For $\alpha\in\mathbb{R}$, we define
\begin{align}\label{eq:importantf}
	f(\alpha)=1-\bar{x}'(\alpha). 
\end{align}
Then 
\begin{align}\label{eq:fpush}
	\bar{x}\#(f\di\alpha)=\bar{\mu},\qquad \|f\|_{L^1}=\bar{\mu}(\mathbb{R}).
\end{align}
\item [(iv)] Let 
\[
A_0^{L,pp}=\{\alpha:~~\bar{x}'(\alpha)=0,~~z_1(\bar{x}(\alpha))<z_2(\bar{x}(\alpha))\},
\]
and
\begin{align*}
A_0^{L,sc}=\{\alpha:~~\bar{x}'(\alpha)=0,~~z_1(\bar{x}(\alpha))=z_2(\bar{x}(\alpha))\},\quad B_0^L=\{\alpha:~~\bar{x}'(\alpha)>0\},
\end{align*}
where $z_1(x)=\inf\{\alpha:~~\bar{x}(\alpha)=x\}$ and $z_2(x)=\sup\{\alpha:~~\bar{x}(\alpha)=x\}.$
Then, we have
\begin{align}\label{eq:decomp}
	\bar{\mu}_{pp}=\bar{x}\#(f|_{A_0^{L,pp}}\di \alpha),\quad \bar{\mu}_{sc}=\bar{x}\#(f|_{A_0^{L,sc}} \di \alpha),\quad \bar{\mu}_{ac}=\bar{x}\#(f|_{B_0^L} \di \alpha).
\end{align}
Moreover,
\begin{align}\label{eq:importantrelation}
	\bar{u}_x^2(\bar{x}(\alpha))\bar{x}'(\alpha)=f(\alpha),\quad \alpha\in B_0^L.
\end{align}

\end{enumerate}

\begin{remark}
\begin{enumerate}
\item [(i)]
In this paper, the superscript $L$ on a set $A^L$ means that we consider the set in the Lagrangian coordinates, while the superscript $E$ on a set $A^E$ means that we consider the set in the Eulerian coordinates.
\item [(ii)]
It is worth noting that $A_0^{L,pp}$ is defined in the point-wise sense,  however, $A_0^{L,sc}$ and $B_0^L$ are defined in the a.e. sense. 
Moreover, when we say some identities hold for $\alpha\in B_0^L$ or $\alpha\in A^{L,sc}_0$ (such as \eqref{eq:importantrelation},  \eqref{eq:falpha} or \eqref{eq:separation} below), we always mean that the identities hold for a.e.  $\alpha\in B_0^L$ or $\alpha\in A^{L,sc}_0$.
Since this will not affect the analysis, we will omit the emphasize of a.e. in  this paper for convenience. Similar treatment also applies to $A_t^{L,pp}$, $A_t^{L,sc}$, and $B_t^L$, which will be defined in Theorem~\ref{thm:measure} below.
\end{enumerate}
\end{remark}

\end{proposition}
\begin{proof}
Since $x\mapsto x+\bar{\mu}((-\infty,x])$ is strictly increasing with possible jumps, $\bar{x}:\mathbb{R}\to\mathbb{R}$ is a non-decreasing function.

(i) Let $\alpha_1,\alpha_2\in\mathbb{R}$ and $\alpha_1<\alpha_2$. According to \eqref{eq:barx1}, we have
\begin{align*}
0\leq \bar{x}(\alpha_2)-\bar{x}(\alpha_1)\leq &\alpha_2-\bar{\mu}((-\infty, \bar{x}(\alpha_2)))-\alpha_1+\bar{\mu}((-\infty, \bar{x}(\alpha_1)])\\
\leq&\alpha_2-\alpha_1-\bar{\mu}((\bar{x}(\alpha_1),\bar{x}(\alpha_2)))\leq\alpha_2-\alpha_1.
\end{align*}
Hence, $\bar{x}$ is Lipschitz continuous with Lipschitz constant bounded by $1$.

(ii) First, we show that definition \eqref{eq:barx1} is equivalent to  definition \eqref{eq:barx2}. Fix $\alpha\in\mathbb{R}$. Let $x_1$ and $x_2$ satisfy
\[
x_1+\bar{\mu}((-\infty, x_1))\leq \alpha \leq x_1+\bar{\mu}((-\infty, x_1]),
\]
and
\[
x_2=\sup\left\{x|x+\bar{\mu}((-\infty,x))\leq \alpha \right\}.
\]
Obviously, we have $x_2\geq x_1$ and $x_2+\bar{\mu}((-\infty,x_2))\leq \alpha$. If $x_1<x_2$, then
\[
\alpha \leq x_1+\bar{\mu}((-\infty, x_1])< x_2+\bar{\mu}((-\infty, x_2)),
\]
which gives a contradiction.

Next, we show that definition \eqref{eq:barx1} is equivalent to  definition \eqref{eq:barx3}. Fix $\alpha\in\mathbb{R}$. Let $x_1$ and $x_3$ satisfy
\[
x_1+\bar{\mu}((-\infty, x_1))\leq \alpha \leq x_1+\bar{\mu}((-\infty, x_1]),
\]
and
\[
x_3=\inf\left\{x|x+\bar{\mu}((-\infty,x])\geq \alpha \right\}.
\]
Obviously, we have $x_3\leq x_1$ and $x_3+\bar{\mu}((-\infty,x_3])\geq \alpha$. If $x_3<x_1$, then
\[
\alpha \geq x_1+\bar{\mu}((-\infty, x_1))> x_3+\bar{\mu}((-\infty, x_3]),
\]
which gives a contradiction.

(iii) It  suffices to show 
\[
\bar{\mu}((a,b))=\int_{\bar{x}^{-1}((a,b))}f\di\alpha,
\]
for any open interval $(a,b)\subset\mathbb{R}$. Let
\[
\alpha_1:=\inf\{\alpha:~~\bar{x}(\alpha)>a\},\quad \alpha_2:=\sup\{\alpha:~~\bar{x}(\alpha)< b\}.
\]
Since $\bar{x}$ is continuous, we have $\bar{x}(\alpha_1)=a$ and $\bar{x}(\alpha_2)=b$. Hence, we have $\bar{x}^{-1}((a,b))=(\alpha_1,\alpha_2)$,
\[
a+\bar{\mu}((-\infty,a))\leq\alpha_1\leq a+\bar{\mu}((-\infty,a]),
\]
and
\[
b+\bar{\mu}((-\infty,b))\leq\alpha_2\leq b+\bar{\mu}((-\infty,b]).
\]
We claim that
\begin{align}\label{eq:claim2}
\alpha_1 = a+\bar{\mu}((-\infty,a]),\quad \alpha_2=b+\bar{\mu}((-\infty,b)).
\end{align}
We only prove the first one in \eqref{eq:claim2} and the other one can be obtained similarly. If 
\[
a+\bar{\mu}((-\infty,a)) \leq\alpha_1 < a+\bar{\mu}((-\infty,a])=:\tilde{\alpha}_1,
\]
then
\[
\bar{x}(\alpha)=a,\quad \alpha\in[\alpha_1,\tilde{\alpha}_1],
\]
and
\[
\bar{x}(\alpha)>a,\quad \alpha>\tilde{\alpha}_1.
\]
This implies a contradiction:
\[
\inf\{\alpha:~~\bar{x}(\alpha)>a\}=\tilde{\alpha}_1>\alpha_1.
\]
Hence, the claim \eqref{eq:claim2} holds. We have
\begin{align*}
\int_{\bar{x}^{-1}(a,b)}f\di\alpha=\int_{\alpha_1}^{\alpha_2}[1-\bar{x}'(\alpha)]\di\alpha=&\alpha_2-\alpha_1-\bar{x}(\alpha_2)+\bar{x}(\alpha_1)\\
=&\bar{\mu}((-\infty,b))-\bar{\mu}((-\infty,a])=\bar{\mu}((a,b)).
\end{align*}

Next, we prove $\|f\|_{L^1}=\bar{\mu}(\mathbb{R})$.
By definition \eqref{eq:barx1} of $\bar{x}(\alpha)$, we have
\[
\bar{\mu}((-\infty, \bar{x}(\alpha)))\leq \alpha-\bar{x}(\alpha) \leq \bar{\mu}((-\infty, \bar{x}(\alpha)]),\quad \alpha\in\mathbb{R}.
\]
It is easy to see that $\bar{x}(\alpha)\to \pm\infty$ as $\alpha\to\pm \infty$. Therefore, 
\begin{align}\label{eq:fact}
\lim_{\alpha\to-\infty}[\alpha-\bar{x}(\alpha)]=0,\quad \lim_{\alpha\to+\infty}[\alpha-\bar{x}(\alpha)]=\bar{\mu}(\mathbb{R}).
\end{align}
This implies $\|f\|_{L^1}=\bar{\mu}(\mathbb{R})$.

(iv) Relations \eqref{eq:decomp} directly follow from Lemma \ref{lmm:keylemma}.
Furthermore, since $\bar{x}\#(f|_{B_0^L}\di  \alpha)=\bar{\mu}_{ac}$, we actually have, for any test function $\varphi\in C_b(\mathbb{R})$,
\[
\int_{\mathbb{R}}\varphi(x)\di \bar{\mu}_{ac}=\int_{\mathbb{R}}\varphi(\bar{x}(\alpha))f|_{B_0^L}(\alpha)\di \alpha.
\]
On the other hand, since $(\bar{u},\bar{\mu})\in\mathcal{D}$, we have $\di \bar{\mu}_{ac} = \bar{u}_x^2\di x$, and hence,
\[
\int_{\mathbb{R}}\varphi(x)\di \bar{\mu}_{ac}=\int_{\mathbb{R}}\varphi(x)\bar{u}_x^2(x)\di x=\int_{\mathbb{R}}\varphi(\bar{x}(\alpha))\bar{u}_x^2(\bar{x}(\alpha))\bar{x}'(\alpha)\di \alpha.
\]
Combining the above two identities, we obtain \eqref{eq:importantrelation}. 
\end{proof}

From Proposition \ref{pro:keypro}, we have
\begin{equation}\label{eq:falpha}
f(\alpha)=\left\{
\begin{aligned}
&\bar{u}_x^2(\bar{x}(\alpha))\bar{x}'(\alpha),\quad \alpha\in B_0^L,\\
&1,\quad \alpha\in A_0^L=A_0^{L,pp}\cup A_0^{L,sc}.
\end{aligned}\right.
\end{equation}
This means we use the increasing function $\bar{x}(\alpha)$ to transform the singular part of $\bar{\mu}$ to some constant part of $f(\alpha)$. For example, if $m\delta_x$ is a pure point part of $\bar{\mu}$, we have $z_1(x)<z_2(x)$ and $\bar{x}(\alpha)=x$ for $\alpha\in [z_1(x),z_2(x)]$. Moreover, we also have $f(\alpha)=1$ for $\alpha\in [z_1(x),z_2(x)]$ and
\[
m=\bar{\mu}(\{x\})=\int_{\bar{x}^{-1}(\{x\})}f(\alpha)\di\alpha=\int_{[z_1(x),z_2(x)]} \di\alpha=z_2(x)-z_1(x).
\]
Hence, the function $\bar{x}$ makes the singular part of $\bar{\mu}$ supported at one point $x$ with mass $m$ uniformly distributed on the interval $[z_1(x),z_2(x)]$ with length $m$.

We have the following important results about global characteristics $y(\alpha,t)$:
\begin{proposition}\label{pro:keypro2}
Let $(\bar{u},\bar{\mu})\in\mathcal{D}$. Let $\bar{x}:\mathbb{R}\to\mathbb{R}$ be defined by \eqref{eq:barx1} and $y(\alpha,t)$ be defined by \eqref{eq:xbarbeta} for $\alpha, ~t\in\mathbb{R}$. Then, for any $t\in\mathbb{R}$, $y(\cdot,t)$ is increasing and we have
\begin{equation}\label{eq:separation}
y_{\alpha}(\alpha,t)=\left\{
\begin{aligned}
&\bar{x}'(\alpha)\left[1+\frac{t}{2}\bar{u}_x(\bar{x}(\alpha))\right]^2,\quad \alpha\in B_0^L,\\
&\frac{t^2}{4},\quad \alpha \in A_0^L.
\end{aligned}
\right.
\end{equation}
Here, the sets $A_0^L$ and $B_0^L$ are the same as in \eqref{eq:falpha}.
\end{proposition}
\begin{proof}
Differentiating \eqref{eq:xbarbeta} with respect to $\alpha$ yields
\[
y_{\alpha}(\alpha,t)=\bar{x}'(\alpha)+\bar{u}_x(\bar{x}(\alpha))\bar{x}'(\alpha)t+\frac{t^2}{4}f(\alpha).
\]
Due to \eqref{eq:falpha} and $\bar{x}'(\alpha)=0$ for $\alpha\in A_0^L$, we have \eqref{eq:separation}.
\end{proof}

Let us end this subsection with some remarks about the flow map, provided that the initial data $\bar{\mu}$ is absolutely continuous with respect to the Lebesgue measure.
\begin{remark}\label{rmk:cumulative}
Consider some initial datum $(\bar{u},\bar{\mu})\in\mathcal{D}$ such that $\di\bar{\mu} = \bar{u}_x^2\di x$.
\begin{enumerate}
\item [(i)]  
The cumulative energy distribution is defined by 
\begin{align}\label{eq:defF}
\bar{F}(x)=\bar{\mu}((-\infty,x))=\bar{\mu}((-\infty,x])= \int_{(-\infty,x)}\bar{u}_x^2(y)\di y.
\end{align}
Obviously,
\begin{align}\label{eq:Ffarfield}
\bar{F}(-\infty)=0~\textrm{ and }~\bar{F}(+\infty)=\|\bar{u}_x\|_{L^2}^2.
\end{align} 
In the $\alpha$-variable, the function $f$ given by \eqref{eq:importantf} plays the role of the energy density in the $\alpha$-variable, and the cumulative energy distribution is given by
\[
\int_{(-\infty,\alpha)}f(\alpha)\di\alpha=\alpha-\bar{x}(\alpha) = \bar{F}(\bar{x}(\alpha)).
\]
\item[(ii)] Since $\bar{\mu}$ is absolutely continuous, there is no need to use the variable $\alpha$ to define the cumulative energy distribution $\bar{F}$. Indeed, we could also use the usual flow map to define the characteristics, which is defined by 
\begin{equation}\label{eq:Lagrange}
\left\{
\begin{aligned}
&\frac{\partial}{\partial t}X(\xi,t)=u(X(\xi,t),t),~~\xi\in\mathbb{R},~~t>0,\\
&X(\xi,0)=\xi.
\end{aligned}
\right.
\end{equation}
Taking time derivative again and using the energy conservation condition, we could formally derive the following global trajectories by similar method to \eqref{eq:global1} or \eqref{eq:global2}:
\begin{align}\label{eq:X}
X(\xi,t)=\xi+\bar{u}(\xi)t+\frac{t^2}{4} \bar{F}(\xi).
\end{align}
We will discuss about the relation between $y(\alpha,t)$ (given by \eqref{eq:xbarbeta}) and $X(\xi,t)$ in the next subsection.

\end{enumerate}

\end{remark}

\subsection{Structure of solutions} 
In this subsection and in particularly, Theorem \ref{thm:measure}, we will first define $u$ and $\mu$ via \eqref{eq:measuresolutionu} and \eqref{eq:measuresolutionmu} for a given initial datum $(\bar{u}, \bar{\mu})$ as well as study their properties;  we state and prove our main results for the structure of  $({u}, {\mu})$ using the variable $\alpha$, which covers Theorem \ref{thm:introduce} and Theorem \ref{thm:2}, except that, we will postpone to show that  they are a global-in-time conservative solution to the generalized framework \eqref{eq:gHS1}-\eqref{eq:gHS3} in the sense of Definition \ref{def:weak} with initial datum $(\bar{u},\bar{\mu})$ in Section~\ref{sec:existenceuniqueness}, see  Theorem~\ref{thm:main} . 
\begin{theorem}\label{thm:measure}
Let $(\bar{u},\bar{\mu})\in\mathcal{D}$,  $\bar{x}(\alpha)$, and $f(\alpha)$  be the same as in Proposition \ref{pro:keypro}. Let $y(\alpha,t)$ be defined by \eqref{eq:xbarbeta}, and 
$z_1(x,t)=\inf\{\alpha:~~y(\alpha,t)=x\},~~z_2(x,t)=\sup\{\alpha:~~y(\alpha,t)=x\}$ be two pseudo-inverses of $y(\cdot,t)$ for a fixed $t\in\mathbb{R}$. For any fixed $t\in\mathbb{R}$, let 
\[
A_t^{L,pp}=\{\alpha:~~y_\alpha(\alpha,t)=0,~~z_1(y(\alpha,t),t)<z_2(y(\alpha,t),t)\},
\]
\begin{align*}
A_t^{L,sc}=\{\alpha:~~y_\alpha(\alpha,t)=0,~~z_1(y(\alpha,t),t)=z_2(y(\alpha,t),t)\},\quad\mbox{and}\quad B_t^L=\{\alpha:~~y_\alpha(\alpha,t)>0\}.
\end{align*}
For any $t\in\mathbb{R}$, define 
\begin{align}\label{eq:measuresolutionu}
u(x,t)=\frac{\partial}{\partial t}y(\alpha,t)=\bar{u}(\bar{x}(\alpha))+\frac{t}{2}(\alpha-\bar{x}(\alpha))~\textrm{ for }~ x=y(\alpha,t),
\end{align}
and
\begin{equation}\label{eq:measuresolutionmu}
\mu(t)=y(\cdot,t)\# (f\di \alpha).
\end{equation}
Let $\mu(t)=\mu_{ac}(t)+\mu_{pp}(t)+\mu_{sc}(t)$, where $\mu_{ac}(t)$, $\mu_{pp}(t)$, and $\mu_{sc}(t)$ are the absolutely continuous part, pure point part and the singular continuous part of ${\mu}(t)$ respectively.
Then we have
\begin{enumerate}
\item [$\bullet$] \textbf{Properties of $\mu(t)$:}
\item [(i)] Energy conservation: we have $\mu(t) \in C(\mathbb{R}; \mathcal{M}_+(\mathbb{R}))$ and
\begin{align}\label{eq:conser}
\mu(t)(\mathbb{R})= \bar{\mu}(\mathbb{R}),\quad t\in\mathbb{R}.
\end{align}
\item [(ii)] 
For the structure of $\mu(t)$, we have
\begin{align}\label{eq:decompmut}
\mu_{pp}(t)=y(\cdot,t)\#(f|_{A_t^{L,pp}}\di \alpha),\quad \mu_{sc}(t)=y(\cdot,t)\#(f|_{A_t^{L,sc}}\di \alpha),\quad \mu_{ac}(t)=y(\cdot,t)\#(f|_{B_t^L}\di \alpha).
\end{align}
Moreover, for any $t\neq 0$, we have
\begin{align}\label{eq:singulart}
A_t^{L,pp}\cup A_t^{L,sc}=\left\{\alpha\in B_0^L:~~\bar{u}_x(\bar{x}(\alpha))=-\frac{2}{t}\right\}:=A^L_t,
\end{align}
where $B_0^L$ was defined in Proposition \ref{pro:keypro} (iv). This implies that for any $t\neq0$, $\mu_{pp}(t)$ and $\mu_{sc}(t)$ are determined by the absolutely continuous part of $\bar{\mu}_{ac}$, i.e., $\bar{u}_x$. Moreover,
\[
\mathrm{supp}\Big(\mu_{pp}(t)+\mu_{sc}(t)\Big)\subset \Big\{x+\bar{u}(x)t+\frac{t^2}{4}\bar{\mu}((-\infty,x)):~ x\in A_t^E\Big\},
\]
where $A_t^E$ is defined by \eqref{eq:singulartE}.
\item [(iii)] There are at most countably many time $t\in \mathbb{R}$  such that either the pure point part or the singular continuous part of $\mu(t)$ is not zero; in other words, both the sets
\begin{align}\label{eq:singulartime}
T_p:=\{t:~\mu_{pp}(t)\neq 0\}\quad\mbox{and}\quad	
T_s:=\{t:~\mu_{sc}(t)\neq 0\}
\end{align}
are countable. 

\

\item [$\bullet$] \textbf{Properties of $u(x,t)$:}
\item [(iv)] For all time $t\in\mathbb{R}$, the function $u(\cdot,t)$ is globally absolutely continuous and  
\begin{align}\label{eq:acpart}
\di\mu_{ac}(t)=u_x^2(x,t)\di x.
\end{align}
Moreover,
\begin{align}\label{eq:propertiesu}
u\in C(\mathbb{R};C_b(\mathbb{R}))\cap C^{1/2}_{loc}(\mathbb{R}\times\mathbb{R}),~~ u_x\in L^\infty(\mathbb{R};L^2(\mathbb{R})),~~ u_t\in L_{loc}^2(\mathbb{R}\times \mathbb{R}).
\end{align}
\item [(v)] If $\bar{u}(-\infty):=\lim_{x\to-\infty}\bar{u}(x)$ exists, then we  have
\begin{align}\label{eq:far}
	\lim_{x\to-\infty}u(x,t)= \bar{u}(-\infty).
\end{align}
On the other hand, if $\bar{u}(+\infty):=\lim_{x\to+\infty}\bar{u}(x)$ exists, then we also have
\begin{align}\label{eq:rightfar}
	\lim_{x\to+\infty}u(x,t)= \bar{u}(+\infty)+\frac{1}{2}\bar{\mu}(\mathbb{R})t.
\end{align}

\item [$\bullet$] \textbf{Relations with the absolutely continuous part:}
\item [(vi)] Consider a time $s\in\mathbb{R}$ such that $\mu(s)$ is absolutely continuous with respect to the Lebesgue measure. Let $\tilde{u}(x)=u(x,s)$, and $X(\xi,t)$ and $\tilde{F}$  be defined by \eqref{eq:X} and \eqref{eq:defF} corresponding to $\tilde{u}$. Then, we have
\begin{align}\label{eq:propertiestilteu}
\tilde{u}\in C_b(\mathbb{R}), \quad \tilde{u}_{x}\in L^2(\mathbb{R}),\quad \| \tilde{u}_{x}\|_{L^2}=\bar{\mu}(\mathbb{R}).
\end{align}
For any $t\in\mathbb{R}$, we also have
\begin{align}\label{eq:composition}
y(\cdot,t)=X(\cdot,t-s)\circ y(\cdot,s), 
\end{align}
\begin{align}\label{eq:startfromL2}
\mu(t)=X(\cdot,t-s)\#(\tilde{u}_x^2\di x),
\end{align}
and
\begin{align}\label{eq:startfromL22}
u(x,t)=\frac{\partial}{\partial t}X(\xi,t-s)=\tilde{u}(\xi) +\frac{(t-s)}{2}\tilde{F}(\xi),\quad \textrm{for $x=X(\xi,t-s)$} .
\end{align}
\end{enumerate}
\end{theorem}

\begin{proof}
(i) 
For the continuity of measure $\mu(t)$, we take any bounded continuous function $\phi(x)$ and using Lebesgue dominated convergence theorem to obtain
\begin{equation*}
\lim_{s\to t}\int_{\mathbb{R}}\phi(x)\di \mu(s) =\lim_{s\to t} \int_{\mathbb{R}}\phi(y(\alpha,s)) f(\alpha) \di \alpha = \int_{\mathbb{R}}\phi(y(\alpha,t)) f(\alpha) \di \alpha =\int_{\mathbb{R}}\phi(x)\di \mu(t).
\end{equation*}	
The conservation of total variation  of $\mu(t)$ (i.e.  \eqref{eq:conser}) follows directly from \eqref{eq:fpush}.

(ii) From Lemma \ref{lmm:keylemma}, we have decomposition \eqref{eq:decompmut}.
 
From \eqref{eq:separation}, we know that the singular part of $\mu(t)$ comes from the set \eqref{eq:singulart}. Hence, all the singular parts of $\mu(t)$ for $t\neq0$ are determined by the function $\bar{u}_x$, or equivalently, the absolutely continuous part of $\bar{\mu}$. 
More precisely, the singular parts come from  $A_t^E$ defined by \eqref{eq:singulartE}, and
\[
\mathrm{supp}\Big(\mu_{pp}(t)+\mu_{sc}(t)\Big)\subset y(A_t^L,t)=\Big\{x+\bar{u}(x)t+\frac{t^2}{4}\bar{\mu}((-\infty,x)):~ x\in A_t^E\Big\}.
\]

(iii)  
Since $A^{L,pp}_t$ consists of closed intervals with positive length, we can always pick a rational number from each interval. Due to  $\bar{u}_x(\bar{x}(\alpha))=-\frac{2}{t}$ for all $\alpha\in A^{L,pp}_t$, for different $t\in T_p$, the chosen rational numbers are different because the closed intervals at different times are different. Therefore, using the chosen rational numbers as indices of these intervals, we know that there are at most countable numbers in $T_p$. 

Let $t\in T_s$. Since $\mu_{sc}(t)\neq 0$, it follows from \eqref{eq:decompmut} that $\mathcal{L}(A_t^{L,sc})> 0$. Furthermore, for $t_1,t_2\in T_s$ with $t_1\neq t_2$, we have $A_{t_1}^{L,sc}\cap A_{t_2}^{L,sc}=\emptyset$, since  $\bar{u}_x(\bar{x}(\alpha))=-\frac{2}{t}$ for all  $\alpha\in A^{L,sc}_t$.  
We obtain the desired result by applying Lemma \ref{lmm:twoclaims} (i).

(iv) $\bullet$ Proof of the local absolutely continuity of $u(\cdot,t)$ for $t\in\mathbb{R}$.

We first prove that $u(\cdot,t)$ is locally absolutely continuous.
Consider $m$ non-overlapping intervals $\{(x_k,y_k)\}_{k=1}^m$ contained in $(-R, R)$ for some $R>0$. Let $x_k=y(\alpha_k,t)$ and $y_k=y(\beta_k,t)$. Using \eqref{eq:measuresolutionu} and the fundamental theorem of calculus, we have
\begin{align*}
\sum_{k=1}^m|u(y_k,t)-u(x_k,t)|=&\sum_{k=1}^m\left|\bar{u}(\bar{x}(\beta_k))+\frac{t}{2}(\beta_k-\bar{x}(\beta_k)) -\bar{u}(\bar{x}(\alpha_k))-\frac{t}{2}(\alpha_k-\bar{x}(\alpha_k))\right|\\
=&\sum_{k=1}^m\left|\int_{[\alpha_k,\beta_k]}\bar{u}_x(\bar{x}(\alpha))\bar{x}'(\alpha)+\frac{t}{2}f(\alpha)\di\alpha\right|.
\end{align*}
Notice that on $A_t^L$ given by \eqref{eq:singulart}, equality \eqref{eq:importantrelation} holds and we have
\begin{align}\label{eq:importantrelation1}
\bar{u}_x(\bar{x}(\alpha))\bar{x}'(\alpha)+\frac{t}{2}f(\alpha)=\bar{u}_x(\bar{x}(\alpha))\bar{x}'(\alpha)\left[1+\frac{t}{2}\bar{u}_x(\bar{x}(\alpha))\right]=0.
\end{align}
Define $D^L=\cup_{k=1}^m[\alpha_k,\beta_k]\setminus A_t^L$. Then, we have
\begin{align}\label{eq:AClebesgue}
\sum_{k=1}^m|u(y_k,t)-u(x_k,t)|\leq \int_{D^L}\left|\bar{u}_x(\bar{x}(\alpha))\bar{x}'(\alpha)+\frac{t}{2}f(\alpha)\right|\di\alpha.
\end{align}
Next, we will prove that for any $\eta>0,$ there exists $\delta>0$ such that if $\sum_{k=1}^m(y_k-x_k)<\delta$, then
\[
\mathcal{L}(D^L)<\eta.
\]
This will prove the absolute continuity of $u(\cdot,t)$ due to the absolute continuity of the Lebesgue integral on the right hand side of \eqref{eq:AClebesgue}.
For this purpose, we define
\[
D^E=\cup_{k=1}^m(x_k,y_k)\setminus y(A_t^L,t).
\]
 Then $z_1(x,t)=z_2(x,t)$ on $D^E$. Moreover, $z_{1x}(\cdot,t)$ exists on $D^E$ and
\[
0<z_{1x}(x,t)=1/y_\alpha(\alpha,t)<\infty,\quad x\in D^E,\quad x=y(\alpha,t).
\]
Let
\[
g(x)=\left\{
\begin{aligned}
&z_{1x}(x,t),~~x\in \mathbb{R} \backslash y(A_t^L,t)\\
&0,~~x\in y(A_t^L,t).
\end{aligned}
\right.
\]
Then, obviously $g\in L^1_{loc}(\mathbb{R})$, since $\int_{(a,b)} g(x) \di x \le \int_{(a,b) }  z_{1x}(x,t) \di x \le z_{1}(b, t) -z_1(a, t)<\infty$ for any $b>a$.
We have the following relation:
\begin{align*}
\mathcal{L}(D^L)=\mathcal{L}(z_1(D^E,t))\leq \int_{D^E}z_{1x}(x,t)\di x= \sum_{k=1}^m\int_{[x_k,y_k]} g(x)\di x,
\end{align*}
where we used \cite[Eq. (4.41)]{benedetto2010integration} for the inequality.
Since $g\in L^1(K)$ for any compact set $K\subset\mathbb{R}$, if $\sum_{k=1}^m(y_k-x_k)\subseteq K$ is sufficiently small (depending on the compact set $K$), then the measure $\mathcal{L}(D^L)$ will be smaller than $\eta$. This shows that $u(\cdot,t)$ is locally absolutely continuous. Moreover, since $u_x(\cdot, t) \in L^2(\mathbb{R})$ (which will be shown below), we can apply Lemma 
\ref{lem:locally and global absolutely continuous} (with $p$ =2) to conclude that $u(\cdot, t)$ is globally absolutely continuous.

$\bullet$ Proof of \eqref{eq:acpart}. 

Since $u(\cdot,t)$ is absolutely continuous, differentiating $u(y(\alpha,t),t)=\bar{u}(\bar{x}(\alpha))+\frac{t}{2}(\alpha-\bar{x}(\alpha))$ with respect to $\alpha$, and then taking a square, we obtain from \eqref{eq:falpha} that
\[
u^2_x(y(\alpha,t),t)y^2_\alpha(\alpha,t)=[\bar{u}_x(\bar{x}(\alpha))\bar{x}'(\alpha)]^2\left[1+\frac{t}{2}\bar{u}_x(\bar{x}(\alpha))\right]^2 \textrm{for $\alpha\in B_0^L$}.
\]
Due to \eqref{eq:separation} and \eqref{eq:falpha},
 we have
\begin{align*}
u^2_x(y(\alpha,t),t)y_\alpha(\alpha,t) =\bar{u}_x^2(\bar{x}(\alpha))\bar{x}'(\alpha) =f(\alpha),~~\alpha\in B^L_t \cap B^L_0.
 \end{align*}

On the other hand, since $\bar{x}'(\alpha) = 0$ on $A^L_0$, we have 
\[
u_x(y(\alpha,t),t)y_\alpha(\alpha,t)= \bar{u}_x(\bar{x}(\alpha))\bar{x}'(\alpha) + \frac{t}{2}(1-\bar{x}'(\alpha)) = \frac{t}{2} ~\textrm{for $\alpha \in A^L_0$}.
\]
Using \eqref{eq:separation} again, we conclude that $u_x(y(\alpha,t),t) = \frac{2}{t}$ on $A^L_0 \subset B^L_t$. Hence, on $A^L_0$, \[
u_x^2(y(\alpha,t),t)y_\alpha(\alpha,t)= \frac{4}{t^2}\frac{t^2}{4} =1 = f(\alpha).
\]
In summary, we have 
\begin{align}\label{eq:acpart2}
u^2_x(y(\alpha,t),t)y_\alpha(\alpha,t)  =f(\alpha),~~\alpha\in B^L_t.
\end{align}
Since $\mu_{ac}(t)=y(\cdot,t)\# (f|_{B^L_t}\di \alpha)$ and $\mathcal{L}(\mathbb{R}\setminus y(B_t^L,t))=0$ (see Lemma \ref{lmm:keylemma}), relation \eqref{eq:acpart} holds.
$\bullet$ Proof of \eqref{eq:propertiesu}.

From the above proof, we have $u\in C(\mathbb{R};C_b(\mathbb{R}))$ and $u_x\in L^\infty(\mathbb{R};L^2(\mathbb{R}))$.
From the definition of $u$, we also have $u_t\in L_{loc}^2(\mathbb{R}\times\mathbb{R})$.
To prove \eqref{eq:propertiesu}, it suffice to show  $u\in C_{loc}^{1/2}(\mathbb{R}\times\mathbb{R})$. Since we have, for any $x\neq y$ and $t\neq s$,
\begin{align*}
	\frac{|u(x,t) -u(y,s)|}{{(|x-y|^2 + |t-s|^2)}^{\frac{1}{4}}} &\leq C\frac{|u(x,t) -u(y,s)|}{|x-y|^\frac{1}{2} + |t-s|^\frac{1}{2}}\\
	&\le C\frac{|u(x,t) -u(y,t)|}{|x-y|^\frac{1}{2}} + C\frac{|u(y,t) -u(y,s)|}{|t-s|^\frac{1}{2}}
\end{align*}
for some constant $C>0$, it is enough to show that $u$ is locally H\"older continuous with order $\frac{1}{2}$ in spatial and temporal variable respectively.
Since $u(\cdot, t) \in H^1_{loc}(\mathbb{R})$ for all time $t\in\mathbb{R}$, we have $u(\cdot, t) \in C^{1/2}_{loc}(\mathbb{R})$. To show the local temporal H\"older continuity, we choose an $\alpha$ such that $y(\alpha,t) = x$ for a given $x \in \mathbb{R}$.  Then using the  spatial H\"older continuity and the definition of $u$, we have 
\begin{align*}
	|u(x,t) -u(x,s)|&\le  |u(y(\alpha,t),t) - u(y(\alpha,s),s)| + |u(y(\alpha,s),s) -u(y(\alpha,t),s)|\\
	&\le \frac{|\alpha-\bar{x}(\alpha)|}{2}|t-s| + C|y(\alpha,s) -y(\alpha,t)|^{\frac{1}{2}}\leq C |t-s|^{\frac{1}{2}}.
\end{align*}

(v) 
This follows from \eqref{eq:fact} and definition \eqref{eq:measuresolutionu}.

(vi) Notice that \eqref{eq:propertiestilteu} follows directly from (i) and (iv). 

For this particular $s$, $\mu(s)=y(\cdot,s)\# (f\di \alpha)$ is absolutely continuous, we have $y_\alpha(\alpha,s)>0$ for a.e. $\alpha\in \mathbb{R}$. Due to \eqref{eq:separation}, this is equivalent to $\bar{u}_x(x)\neq -\frac{2}{s}$ for a.e. $x\in\mathbb{R}$.
From the definitions, we have
\begin{align}\label{eq:sub1}
y(\alpha,s)=\bar{x}(\alpha)+\bar{u}(\bar{x}(\alpha))s+\frac{s^2}{4}(\alpha-\bar{x}(\alpha)),~~\alpha\in\mathbb{R},
\end{align}
and
\begin{align}\label{eq:sub2}
\tilde{u}(x)=u(x,s)=\bar{u}(\bar{x}(\alpha))+\frac{s}{2}(\alpha-\bar{x}(\alpha))~\textrm{ for }~ x=y(\alpha,s).
\end{align}
By (ii) and (iii), we have $\tilde{u}_x^2 \di x=y(\cdot,s)\# (f\di \alpha)$.  Hence,
\begin{align}\label{eq:sub3}
\tilde{F}(y(\alpha,s))=\int_{(-\infty,y(\alpha,s))}\tilde{u}_x^2(x)\di x=\int_{(-\infty,\alpha)} f(\alpha)\di \alpha=\alpha-\bar{x}(\alpha).
\end{align}
By \eqref{eq:X}, we have
\[
X(\xi,t-s)=\xi+\tilde{u}(\xi)(t-s)+\frac{(t-s)^2}{4}\tilde{F}(\xi),~~\xi\in\mathbb{R}.
\]
Hence, combining \eqref{eq:sub1}, \eqref{eq:sub2}, and \eqref{eq:sub3} yields \eqref{eq:composition}: for any $\alpha$, $t\in\mathbb{R}$,
\begin{align*}
X(y(\alpha,s),t-s)&=y(\alpha,s)+\tilde{u}(y(\alpha,s))(t-s)+\frac{(t-s)^2}{4}\tilde{F}(y(\alpha,s))\\
&=\bar{x}(\alpha)+\bar{u}(\bar{x}(\alpha))s+\frac{s^2}{4}(\alpha-\bar{x}(\alpha))+\left[\bar{u}(\bar{x}(\alpha))+\frac{s}{2}(\alpha-\bar{x}(\alpha))\right](t-s)\\
&\quad+\frac{(t-s)^2}{4}(\alpha-\bar{x}(\alpha))\\
&=\bar{x}(\alpha)+\bar{u}(\bar{x}(\alpha))t+\frac{t^2}{4}(\alpha-\bar{x}(\alpha))=y(\alpha,t).
\end{align*}
For \eqref{eq:startfromL2}, we have 
\[
\mu(t)=y(\cdot,t)\# (f\di \alpha)=[X(\cdot,t-s)\circ y(\cdot,s)]\# (f\di \alpha)=X(\cdot,t-s)\#[y(\cdot,s)\# (f\di \alpha)]=X(\cdot,t-s)\#(\tilde{u}_x^2\di \alpha).
\]

Next, we show \eqref{eq:startfromL22}. For $x=X(\xi,t-s)$, consider $\alpha$ satisfying $\xi=y(\alpha,s)$ and then 
\[
x=X(y(\alpha,s),t-s)=y(\alpha,t).\]
We have
\begin{align*}
\tilde{u}(\xi) +\frac{(t-s)}{2}\tilde{F}(\xi)&= \bar{u}(\bar{x}(\alpha))+\frac{s}{2}(\alpha-\bar{x}(\alpha))+\frac{(t-s)}{2}(\alpha-\bar{x}(\alpha))\\
&=\bar{u}(\bar{x}(\alpha))+\frac{t}{2}(\alpha-\bar{x}(\alpha))=u(x,t).
\end{align*}

\end{proof}
Let us end this section by some remarks. 
\begin{remark}\label{rmk:thm21}	
\begin{enumerate}
\item[(i)] (Relation with regular initial data) For a non-smooth initial datum $(\bar{u},\bar{\mu})\in\mathcal{D}$, let $(u,\mu)$ be a solution to the generalized framework \eqref{eq:gHS1}-\eqref{eq:gHS3} defined by \eqref{eq:measuresolutionu} and \eqref{eq:measuresolutionmu}.
Then $(u,\mu)$ can also be constructed by a regular function $\tilde{u}(x)$ as given in Theorem~\ref{thm:measure} (vi). More precisely, such a regular function $\tilde{u}(x):=u(x,s)$ is obtained by evaluating $u$ at a time $s$ such that $\mu$ is absolutely continuous. This time $s$ always exists due to Theorem~\ref{thm:measure} (iii). Start from $\tilde{u}$ and we can use the traditional flow map $X(\xi,t)$ in the Lagrangian coordinates to construct solution $(v,\nu)$. Then $(u,\mu)$  will be recovered by some time shifting of $(v,\nu)$.

\item[(ii)] (Separation of singular part and absolutely continuous part of $\mu(t)$) In Proposition \ref{pro:keypro} (iv), we used $\bar{x}(\alpha)$ to identify the singular parts (corresponding to $A_0^L=A_0^{L,pp}\cup A_0^{L,sc}$) and absolutely continuous part of $\bar{\mu}$ (corresponding to $B_0^L$). 
Since the singular part of $\bar{\mu}$  comes from $A_0^L$, we can conclude from Theorem \ref{thm:measure} (ii) (or \eqref{eq:singulart}) that this singular part will never create singular parts of $\mu(t)$ again for any other $t$, and all the singular parts of $\mu(t)$ ($t\neq 0$) are generated by the absolutely continuous part of $\bar{\mu}$, i.e. $\bar{u}_x$. We can also see this from the fact that  $y_\alpha(\alpha,t)=\frac{t^2}{4}>0$ for $\alpha\in A_0^L$ (see \eqref{eq:separation}). 
\item[(iii)] (Calculation of mass for singular part of $\mu(t)$) Due to $f(\alpha)=1-\bar{x}'(\alpha)$, we deduce from \eqref{eq:importantrelation} that $\bar{x}'(\alpha) = \frac{t^2}{t^2+4}$ and $f(\alpha)=\frac{4}{t^2+4}$ on $A^{L}_t$. We consider a point $x_0 \in y( A_t^{L,pp},t)$ and $\alpha_1:=z_1(x_0,t)<z_2(x_0,t)=:\alpha_2$. Then, $\bar{u}_x(x)= -\frac{2}{t}$ for all $x\in[x_1,x_2]:=[\bar{x}(\alpha_1),\bar{x}(\alpha_2)]$.
The mass concentrated at $x$ is calculated by
\begin{align*}
\mu(t)(\{x_0\})=\int_{[\alpha_1,\alpha_2]}f(\alpha)\di \alpha = \frac{4}{t^2+4}(\alpha_2-\alpha_1)=\frac{4}{t^2}(x_2-x_1).
\end{align*}
For the singular continuous part,  define $A_t^{E,sc}=\bar{x}(A_t^{L,sc})$ and then $\bar{u}_x(x)=-\frac{2}{t}$ for $x\in A_t^{E,sc}$. We have
\begin{align*}
\mu_{sc}(t)(\mathbb{R})=\frac{4}{t^2}\mathcal{L}(A_t^{E,sc}).
\end{align*}
See Example \ref{ex:singularremark} for a detailed calculation.

\item[(iv)] (Luzin N property)
The absolutely continuous result stated in Theorem \ref{thm:measure} (iv) is more straightforward when $t\notin T_s\cup T_p$. This is because $y(\cdot,t)$ is strictly increasing when $t\notin T_s\cup T_p$, and there exists a unique absolutely continuous inverse of $y(\cdot,t)$ by (ii) in Lemma \ref{lmm:twoclaims}. Using this inverse, we can directly prove that $u(\cdot,t)$ is absolutely continuous on any bounded interval. Moreover, we have
\[
\int_{\mathbb{R}}u_{x}^2(x,t)\di x=\mu_{ac}(t)(\mathbb{R})=\mu(t)(\mathbb{R})=\bar{\mu}(\mathbb{R}),\quad\mbox{for all } t\notin T_s\cup T_p.
\]
However, for $t\in T_s\cup T_p$, the inverse of $y(\cdot,t)$ does not satisfy the Luzin N property on the set $A_t^{L,pp}\cup A_t^{L,sc}$ (see Example~\ref{ex:singularremark} for instance) and it is not absolutely continuous. In this case, the function $u(\cdot,t)$ is still absolutely continuous, and we have
\[
\int_{\mathbb{R}}u_{x}^2(x,t)\di x=\mu_{ac}(t)(\mathbb{R})<\mu(t)(\mathbb{R})=\bar{\mu}(\mathbb{R}),\quad\mbox{for all } t\in T_s\cup T_p.
\]
The above analysis also shows that usually $u_x\notin C(\mathbb{R};L^2(\mathbb{R}))$.

\end{enumerate}

\end{remark}

\section{Existence and uniqueness  of conservative solutions}\label{sec:existenceuniqueness}
In this section, we are going to show that $(u,\mu)$ defined by \eqref{eq:measuresolutionu}-\eqref{eq:measuresolutionmu} is a conservative solution to the system \eqref{eq:gHS1}-\eqref{eq:gHS3}. We will also show the uniqueness of conservative solutions via characteristics method.
\subsection{Existence}
We have the following existence theorem:
\begin{theorem}[Existence]\label{thm:main}
Let $(\bar{u},\bar{\mu})\in \mathcal{D}$ be an initial datum. Let $u$ be defined by \eqref{eq:measuresolutionu} and $\mu$ be defined by \eqref{eq:measuresolutionmu}. Then, $(u(t), \mu(t))$ is a global-in-time conservative solution to the generalized framework \eqref{eq:gHS1}-\eqref{eq:gHS3} in the sense of Definition \ref{def:weak} with initial datum $(\bar{u},\bar{\mu})$. Moreover, the function $u$ and energy measure $\mu$ satisfy all the properties in Theorem \ref{thm:measure}.
\end{theorem}

\begin{proof}
	It follows from Theorem~\ref{thm:measure}, Equation \eqref{eq:measuresolutionu} and \eqref{eq:measuresolutionmu} that $(u(t), \mu(t))$ satisfy properties (i), (ii) and (v) of Definition~\ref{def:weak}.
We are going to use the change of variables  $y(\alpha,t)$ to prove (iii) and (iv) (i.e., \eqref{eq:weakformula} and  \eqref{eq:fourth}) of Definition~\ref{def:weak}. Since $u_t$ has enough regularity and $y(\alpha,t)$ is an absolutely continuous function of $\alpha$, we have, for $\phi\in C_c^\infty(\mathbb{R}\times\mathbb{R})$,
\begin{align*}
&\int_{\mathbb{R}}\int_{\mathbb{R}}u\phi_t \di x\di t= -\int_{\mathbb{R}}\int_{\mathbb{R}}u_t\phi \di x\di t  = -\int_{\mathbb{R}}\int_{\mathbb{R}}u_t (y(\alpha,t),t) \phi(y(\alpha,t), t) y_{\alpha}(\alpha,t)\di \alpha\di t \\
=& -\int_{\mathbb{R}}\int_{\mathbb{R}}\frac{\di}{\di t}  u(y(\alpha,t),t) \phi(y(\alpha,t), t) y_{\alpha}(\alpha,t)\di \alpha\di t +  \int_{\mathbb{R}}\int_{\mathbb{R}}(uu_x)(y(\alpha,t), t) \phi(y(\alpha,t), t) y_{\alpha}(\alpha,t)\di \alpha\di t \\
=& -\int_{\mathbb{R}}\int_{\mathbb{R}}\frac{\di }{\di t}  u(y(\alpha,t),t) \phi(y(\alpha,t), t) y_{\alpha}(\alpha,t)\di \alpha\di t  +  \int_{\mathbb{R}}\int_{\mathbb{R}}(uu_x)(x, t) \phi(x, t) \di x\di t.  \\
\end{align*}

Comparing with equation \eqref{eq:weakformula}, we are left to show 
\begin{equation}\label{eq:calF}
\begin{aligned}
\int_{\mathbb{R}}\int_{\mathbb{R}}\frac{\di }{\di t} u(y(\alpha,t),t) \phi(y(\alpha,t), t) y_{\alpha}(\alpha,t)\di \alpha\di t  &= \frac{1}{2}\int_{\mathbb{R}}\int_{\mathbb{R}} \phi(x,t) F(x,t) \di x\di t\\
&= \frac{1}{2}\int_{\mathbb{R}}\int_{\mathbb{R}} \phi(y(\alpha,t),t) F(y(\alpha,t),t) y_{\alpha}(\alpha,t)\di \alpha\di t.
\end{aligned}
\end{equation}

Notice that from the definition of $u$ (see \eqref{eq:measuresolutionu}), we have
$\frac{\di }{\di t}u(y(\alpha,t),t) = \frac{1}{2}(\alpha - \bar{x}(\alpha))$ and the integral on the left hand side of \eqref{eq:calF} becomes
\[
\int_{\mathbb{R}}\int_{\mathbb{R}}\frac{\di }{\di t} u(y(\alpha,t),t) \phi(y(\alpha,t), t) y_{\alpha}(\alpha,t)\di \alpha\di t=\int_{\mathbb{R}}\int_{\mathbb{R}}\frac{1}{2}(\alpha-\bar{x}(\alpha))\phi(y(\alpha,t), t) y_{\alpha}(\alpha,t)\di \alpha\di t.
\]
On the other hand, for all but countably many time $t$, $\mu(t)$ is absolutely continuous. For these time $t$, the absolutely continuity of $\mu(t)$ and \eqref{eq:measuresolutionmu} imply
\begin{align*}
F(y(\alpha,t),t) =  \int_{(-\infty,y(\alpha,t))} \di \mu(t)   
= \int_{(-\infty,\alpha)} f(\alpha)\di \alpha = \alpha -\bar{x}(\alpha).
\end{align*}
This proves \eqref{eq:calF}, and hence, equation \eqref{eq:weakformula} follows.

Next, we prove \eqref{eq:fourth}. Using the definition \eqref{eq:measuresolutionmu} of $\mu(t)$ again, we obtain, for  $\phi\in C_c^\infty(\mathbb{R} \times\mathbb{R})$, 
\begin{equation}\label{eq:conservation of mass}
\begin{aligned}
\int_{\mathbb{R}} \int_{\mathbb{R}} (\phi_t + u\phi_x )\di \mu(t) \di t &= \int_{\mathbb{R}} \int_{\mathbb{R}} \Big[\phi_t (y(\alpha,t), t) + u(y(\alpha,t),t)  \phi_x (y(\alpha,t),t) \Big] f(\alpha)\di \alpha \di t\\
&= \int_{\mathbb{R}} \int_{\mathbb{R}}  \frac{\di }{\di t} \phi(y(\alpha,t), t) f(\alpha)\di \alpha\di t\\
&=\int_{\mathbb{R}} \frac{\di}{\di t} \left( \int_{\mathbb{R}}  \phi(y(\alpha,t), t) f(\alpha)\di \alpha \right)\di t =0.
\end{aligned}
\end{equation}

\end{proof}

\begin{remark}\label{rem: test functions}
	It is obvious to see that \eqref{eq:fourth} still holds for all bounded smooth $\phi(x,t)$ with bounded derivatives supported on $(-\infty, \infty) \times (-M, M)$ for any $M>0$ from the proof of \eqref{eq:conservation of mass}. This property will be used in the proof of Lemma~\ref{lmm:keylemma1}.
\end{remark}

\subsection{Uniqueness}
We will show the uniqueness of conservative solutions via the characteristics method. We use the methods in \cite{bressan2016uniqueness,bressan35uniqueness,bressan2015unique} with some improvements. 
\begin{lemma}\label{lmm:keylemma1}
Let $(v,\nu)$ be a conservative solution to the generalized framework \eqref{eq:gHS1}-\eqref{eq:gHS3} in the sense of Definition \ref{def:weak}. Consider the time $t$ and $\tau$ such that $\nu$ is absolutely continuous. Then, for any fixed $y\in\mathbb{R}$ and $\e_0>0$, 
\begin{align}\label{eq:claim3}
\int_{(-\infty,y+a_-(t-\tau))}v_x^2(x,t)\di x\leq \int_{(-\infty,y)}v_x^2(x,\tau)\di x\leq  \int_{(-\infty,y+a_+(t-\tau))}v_x^2(x,t)\di x,
\end{align}
provided that $t-\tau>0$ is small enough (depending on $v$, $y$ and $\e_0$), where  $a_\pm:=v(y,\tau)\pm\e_0$.  Moreover, for any $T>0$ and any  $-T\leq\tau<t\leq T$,  
\begin{align}\label{eq:claim}
\int_{(-\infty,y-C_T(t-\tau))}v_x^2(x,t)\di x\leq \int_{(-\infty,y)}v_x^2(x,\tau)\di x\leq \int_{(-\infty,y+C_T(t-\tau))}v_x^2(x,t)\di x,
\end{align}
for all $C_T$ satisfying $ \|v\|_{C_{b} (\mathbb{R}\times[-T,T])}\leq C_T$.

\end{lemma}
\begin{proof}
We are going to construct some test functions and use \eqref{eq:fourth} to prove this lemma.

For any fixed $\e>0$ sufficiently small, consider the following two non-negative smooth  functions: 
\begin{equation}\label{eq:theta}
\theta_{\e}(x) =\left\{
\begin{aligned}
&1,\quad x\leq y,\\ 
&0,\quad x\geq y+\e,
\end{aligned}
\right.
\end{equation}
and
$\chi_\e(s)=1$ for $s\in[\tau,t]$ and $\su\chi_\e\subset(\tau-\e,t+\e)$. Assume that 
\[
\theta_\e'\leq 0,\quad \chi'_{\epsilon}(s)>0~\textrm{ for }~s\in(\tau-\e,\tau),\quad \chi'_{\epsilon}(s)<0~\textrm{ for }~s\in(t,t+\e).
\]
Let $\phi^{\pm}_{\e}(x,s)=\theta_{\e}(x-a_\pm(s-\tau))\chi_\e(s)\in C^{\infty}_b(\mathbb{R}\times [0, \infty))$. 
Then,
\begin{align*}
(\phi^{\pm}_{\e t}+v\phi^{\pm}_{\e x})(x,s)=\chi'_\e(s)\theta_\e(x-a_\pm(s-\tau) )+(v-a_\pm)\chi_\e(s)\theta'_\e(x-a_\pm(s-\tau)).
\end{align*}
Since the conservative solution $v$ satisfies \eqref{eq:fourth} and $\phi^{\pm}_{\e}$ is supported on $(-\infty, \infty) \times [\tau-\e,t+\e]$, we can take $\phi^{\pm}_{\e}$ as a test function in  \eqref{eq:fourth} (cf. Remark \ref{rem: test functions}),  and obtain 
\begin{multline}\label{eq:regularity1}
0=\left(\int_{\tau-\e}^{\tau}+\int_{t}^{t+\e}\right)\int_{-\infty}^{y+a_\pm(s-\tau)+\e}\chi'_\e(s)\theta_\e(x-a_\pm(s-\tau))v_x^2(x,s)\di x\di s\\
+\int_{\tau-\e}^{t+\e}\int_{y+a_\pm(s-\tau)}^{y+a_\pm(s-\tau)+\e}(v{ (x,s)}-a_\pm)\chi_\e(s)\theta'_\e(x-a_\pm(s-\tau))v_x^2(x,s)\di x\di s.
\end{multline}

Consider the case for $a_+$.
When $t-\tau$ and $\e$ are sufficiently small, by the continuity of $v$, we have $|v(x,s)-v(y,\tau)|<\e_0$ for $s\in(\tau-\e,t+\e)$ and $x\in (y+a_+(s-\tau), y+a_+(s-\tau)+\e)$, which implies  $v(x,s)-a_+=v(x,s)-v(y,\tau)-\e_0\leq 0$. Hence, the second term in \eqref{eq:regularity1} is positive due to $\theta'_\e\leq0$ and $\chi_\e\geq0$. Therefore,
\begin{align}\label{eq:regu}
\left(\int_{\tau-\e}^{\tau}+\int_{t}^{t+\e}\right)\int_{-\infty}^{y+a_+(s-\tau)+\e}\chi'_\e(s)\theta_\e(x-a_+(s-\tau))v_x^2(x,s)\di x\di s\leq0.
\end{align}
By the definition of $\theta_{\e}$ and $\chi_\e$,  we have
\begin{equation*}
\begin{aligned}
\int_{\tau-\e}^{\tau}\int_{-\infty}^{y+a_+(s-\tau)+\e}\chi'_\e(s)\theta_\e(x-a_+(s-\tau))v_x^2(x,s)\di x\di s
 \ge&\int_{\tau-\e}^{\tau} \chi'_\e(s) \int_{-\infty}^{y+a_+(s-\tau)}  v_x^2(x,s) \di x \di s\\
\geq& \inf_{s\in[\tau - \e, \tau]^0}\int_{-\infty}^{y+a_+(s-\tau)}  v_x^2(x,s) \di x,
\end{aligned}
\end{equation*}
and similarly,
\[
\int_{t}^{t+\e}\int_{-\infty}^{y+a_+(s-\tau)+\e}\chi'_\e(s)\theta_\e(x-a_+(s-\tau))v_x^2(x,s)\di x\di s\geq -\sup_{s\in[t,t+\e]^0}\int_{-\infty}^{y+a_+(s-\tau)+\e} v_x^2(x,s)\di x.
\]
Here, we use $[\tau-\e,\tau]^0$ and $[t,t+\e]^0$ to stand for those times in  $[\tau-\e,\tau]$ and $[t,t+\e]$  such that $\nu$ is absolutely continuous.
As a measure, $\nu$ is continuous in time, so using  \cite[Theorem 1.40]{evans2015measure}, we have
\[
\lim_{\e\to0^+}\int_{\tau-\e}^{\tau}\int_{-\infty}^{y+a_+(s-\tau)+\e}\chi'_\e(s)\theta_\e(x-a_+(s-\tau))v_x^2(x,s)\di x\di s\geq \int_{-\infty}^{y}v_x^2(x,\tau)\di x,
\]
and
\[
\lim_{\e\to0^+}\int_{t}^{t+\e}\int_{-\infty}^{y+a_+(s-\tau)+\e}\chi'_\e(s)\theta_\e(x-a_+(s-\tau))v_x^2(x,s)\di x\di s\geq -\int_{-\infty}^{y+a_+(t-\tau)}v_x^2(x,t)\di x.
\]
Passing to the limit as $\e\to 0^+$ in \eqref{eq:regu} and using the above two inequalities, we obtain the second inequality in \eqref{eq:claim3}.

Consider the case for $a_-$. When $t-\tau$ and $\e$ are sufficiently small, by the continuity of $v$, we have $|v(x,s)-v(y,\tau)|<\e_0$ for $s\in(\tau-\e,t+\e)$ and $x\in (y+a_-(s-\tau), y+a_-(s-\tau)+\e)$, which implies  $v(x,s)-a_-=v(x,s)-v(y,\tau)+\e_0\geq 0$. Hence, the second term in \eqref{eq:regularity1} is negative due to $\theta'_\e\leq0$ and $\chi_\e\geq0$. Therefore,
\begin{align}\label{eq:regu1}
\left(\int_{\tau-\e}^{\tau}+\int_{t}^{t+\e}\right)\int_{-\infty}^{y+a_-(s-\tau)+\e}\chi'_\e(s)\theta_\e(x-a_-(s-\tau))v_x^2(x,s)\di x\di s\geq0.
\end{align}
Similarly to the case for $a_+$, passing to the limit as $\e\to 0^+$ in \eqref{eq:regu1}, we obtain the first inequality in \eqref{eq:claim3}.

One can show \eqref{eq:claim} by applying a similar argument as in the above proof of \eqref{eq:claim3} by using $a_-=-C_T$ and $a_+=C_T$ instead. It is worth noting that $-C_T\leq v\leq C_T$, so the proof will not need the smallness condition for $t-\tau$ because we do not need to apply the continuity of $v$ as in the above proof of \eqref{eq:claim3}.

\end{proof}

Next, we apply the above lemma to prove the following uniqueness theorem. 
\begin{theorem}[Uniqueness of characteristics and conservative solutions]\label{thm:main1}
Let    $(v,\nu)$ be a conservative solution to the generalized framework \eqref{eq:gHS1}-\eqref{eq:gHS3} in the sense of Definition \ref{def:weak} with initial datum $(\bar{u},\bar{\mu}) \in \mathcal{D}$. Then, there exists a unique characteristic $y_1(\alpha,t)$ satisfying
\begin{align}\label{eq:flowmap}
\frac{\partial}{\partial t}y_1(\alpha,t)=v(y_1(\alpha,t),t),\quad y_1(\alpha,0)=\bar{x}(\alpha),
\end{align}
 and
\begin{equation}\label{eq:energyconserved}
\nu(t)((-\infty, y_1(\alpha,t)))\leq \alpha-\bar{x}(\alpha)\leq \nu(t)((-\infty, y_1(\alpha,t)]),
\end{equation}
for any $\alpha \in\mathbb{R}$ and a.e $t\in \mathbb{R}$, where $\bar{x}(\alpha)$ is defined by \eqref{eq:barx1}. 
The uniqueness of characteristics and conservative solutions follows, i.e.,   $(v,\nu) = (u,\mu)$, where $(u,\mu)$ is  defined by \eqref{eq:measuresolutionu} and $\eqref{eq:measuresolutionmu}$. 
\end{theorem}
\begin{proof}
We will separate the proof into three steps.

\textbf{Step 1.} We use a similar definition for $x_1(\beta,t)$  to that of $\bar{x}(\alpha)$: for any $\beta\in\mathbb{R}$ and $t\in \mathbb{R}$,
\begin{align}\label{eq:x1betat}
x_1(\beta,t)+\nu(t)((-\infty, x_1(\beta,t)))\leq \beta\leq x_1(\beta,t)+\nu(t)((-\infty, x_1(\beta,t)]),
\end{align}
which is actually the inverse function of the sum of identity and cumulative energy distribution: $x+\nu(t)((-\infty,x])$. For a.e. $t\in \mathbb{R}$, we have $\di \nu(t)=v_x^2(x,t)\di x$, so at these times, Definition~\eqref{eq:x1betat} is equivalent to
\begin{align}\label{eq:newvariables}
x_1(\beta,t)+\int_{(-\infty, x_1(\beta,t))}v_x^2(y,t)\di y=\beta, 
\end{align}
which implies
\begin{align}\label{eq:xbeta}
|x_{1\beta}(\beta,t)|=\frac{1}{1+v_x^2(x_1(\beta,t),t)}\leq 1,\quad\forall \beta\in\mathbb{R}, \textrm{ a.e. }t\in \mathbb{R}
\end{align}
and
\begin{align}\label{eq:vxbeta}
|\partial_\beta v(x_1(\beta,t),t)|=\left|\frac{v_x(x_1(\beta,t),t)}{1+v_x^2(x_1(\beta,t),t)}\right|\leq \frac{1}{2},\quad\forall \beta\in\mathbb{R}, \textrm{ a.e. }t\in \mathbb{R}.
\end{align}

Next, we prove $t\mapsto x_1(\beta,t)$ is Lipschitz. First, consider $\tau, t \in[-T,T]$ and $\tau<t $ such that $\nu$ is absolutely continuous at $\tau, t$.  Let $y=x_1(\beta,\tau)$. 
It follows from the first inequality of \eqref{eq:claim} that
\begin{align*}
y-C_T(t-\tau)+\int_{(-\infty,y-C_T(t-\tau))} v_x^2(x,t)\di x &\leq x_1(\beta,\tau)+\int_{(-\infty,x_1(\beta,\tau))}v_x^2(x,\tau)\di x=\beta\\
&=x_1(\beta,t)+\int_{(-\infty,x_1(\beta,t))}v_x^2(x,t)\di x,
\end{align*}
which implies $y-C_T(t-\tau)\leq x_1(\beta,t)$, i.e. $x_1(\beta,\tau)-x_1(\beta,t)\leq C_T(t-\tau)$. Similarly, it follows from the second inequality of \eqref{eq:claim} that
\begin{align*}
y+C_T(t-\tau)+\int_{(-\infty,y+C_T(t-\tau))}v_x^2(x,t)\di x&\geq x_1(\beta,\tau)+\int_{(-\infty,x_1(\beta,\tau))}v_x^2(x,\tau)\di x=\beta\\
&=x_1(\beta,t)+\int_{(-\infty,x_1(\beta,t))}v_x^2(x,t)\di x,
\end{align*}
which implies $y+C_T(t-\tau)\geq x_1(\beta,t)$, i.e. $x_1(\beta,t)-x_1(\beta,\tau)\leq C_T(t-\tau)$. Combining the above two results yields
\[
|x_1(\beta,t)-x_1(\beta,\tau)|\leq C_T|t-\tau|,
\]
for any $t$, $\tau$ such that $\nu$ is absolutely continuous. With the above results, when $\nu$ is not absolutely continuous at $t$ and/or $\tau$, we only need to show the continuity of map $t\mapsto x_1(\beta,t)$. Actually, it follows from Definition~\eqref{eq:x1betat} that
\begin{align}
x_1(\beta,t)+\nu(t)((-\infty, x_1(\beta,t)))\leq \beta\leq x_1(\beta,s)+\nu(s)((-\infty, x_1(\beta,s)]),
\end{align}
which implies
\[
 x_1(\beta,s)- x_1(\beta,t)\geq \nu(t)((-\infty, x_1(\beta,t)))-\nu(s)((-\infty, x_1(\beta,s)])
\]
for any $s$, $t\in \mathbb{R}$. Fix a $t \in \mathbb{R}$. We prove the continuity of $x_1(\beta,\cdot)$ at $t$ by a contradiction argument. Seeking for a contradiction, we assume that there exists a sequence $s_k\to t$ as $k\to\infty$ but $\lim_{k\to\infty}x_1(\beta,s_k)=A<x_1(\beta,t)$. Due to Remark \ref{rem: mu(t)(mathbb R) is conserved}, we have 
\begin{align*}
0>&\lim_{k\to \infty}x_1(\beta,s_k)- x_1(\beta,t)\geq  \nu(t)((-\infty, x_1(\beta,t)))-\limsup_{k\to \infty}\nu(s_k)((-\infty, x_1(\beta,s_k)])\\
=&  \nu(t)((-\infty, x_1(\beta,t)))-\limsup_{k\to \infty}[\nu(s_k)(\mathbb{R})-\nu(s_k)(( x_1(\beta,s_k),+\infty))]\\
=&  \nu(t)((-\infty, x_1(\beta,t)))-\bar{\mu}(\mathbb{R})+\liminf_{k\to \infty}\nu(s_k)(( x_1(\beta,s_k),+\infty)).
\end{align*}
For $k$ big enough, we have $x_1(\beta,s_k)<\frac{x_1(\beta,t)+A}{2}<x_1(\beta,t)$, and combining \cite[Theorem 1.40]{evans2015measure} yields a contradiction:
\begin{align*}
0>&\nu(t)((-\infty, x_1(\beta,t)))-\bar{\mu}(\mathbb{R})+\liminf_{k\to \infty}\nu(s_k)\left(\left(\frac{x_1(\beta,t)+A}{2},+\infty\right)\right)\\
\geq&\nu(t)((-\infty, x_1(\beta,t)))-\bar{\mu}(\mathbb{R})+ \nu(t)\left(\left( \frac{x_1(\beta,t)+A}{2},+\infty\right)\right)\geq 0,
\end{align*}
which is a contradiction.
On the other hand, if there is a sequence $\tilde{s}_k\to t$ as $k\to\infty$ but $\lim_{k\to\infty}x_1(\beta,\tilde{s}_k)=B>x_1(\beta,t)$, we can also obtain a contradiction by a similar argument. Combining the above arguments, we know $x_1(\beta,t)$ is continuous with respect to time $t$.

\textbf{Step 2.}  In this step, we are going to prove \eqref{eq:flowmap} and \eqref{eq:energyconserved}.

Consider the integral equation:
\begin{equation}\label{eq:ode1}
\beta_1(t)=\alpha+\int_0^tv(x_1(\beta_1(s),s),s)\di s.
\end{equation}
For $t\in \mathbb{R}$, due to \eqref{eq:vxbeta} and Banach fixed point theorem, there always exists a unique global solution $\beta_1(t)$ to \eqref{eq:ode1}. 
Define
\begin{align}\label{eq:charac}
y_1(\alpha,t)=x_1(\beta_1(t),t),\quad \alpha\in\mathbb{R},~~ t\in \mathbb{R}.
\end{align}
Then, combining \eqref{eq:barx1} and \eqref{eq:x1betat} yields the initial datum: $y_1(\alpha,0)=\bar{x}(\alpha)$. 
From  \eqref{eq:x1betat}, \eqref{eq:ode1}, and \eqref{eq:charac}, the following relation holds: for a.e. $t\in \mathbb{R}$,
\begin{align}\label{eq:combine}
y_1(\alpha,t)+\nu(t)((-\infty, y_1(\alpha,t)))\leq \alpha+\int_0^tv(y_1(\alpha,s),s)\di s\leq y_1(\alpha,t)+\nu(t)((-\infty, y_1(\alpha,t)]).
\end{align}
Fix $\alpha\in \mathbb{R}$. From Step 1, we know that $y_1(\alpha,t)$ is differentiable for a.e. $t>0$.
In the following, we only consider those  $t\in \mathbb{R}$ such that $y_1(\alpha,\cdot)$ is differentiable and $\nu(t)$  is absolutely continuous with $\di \nu(t)=v_x^2(x,t)\di x$. 
We will prove \eqref{eq:flowmap} by a contradiction argument. Seeking for a contradiction, we assume that there exists a $\tau\in\mathbb{R}$ such that $\di \nu(\tau)=v_x^2(x,\tau)\di x$ and $y_1(\alpha,\tau)$ is differentiable but $\partial_ty_1(\alpha,\tau)\neq v(y_1(\alpha,\tau),\tau)$. Then, there exists some $\e_0>0$ such that either one of the following two cases holds:
\begin{enumerate}
\item [\textbf{Case 1:}]
\begin{align}\label{eq:contradiction}
\partial_ty_1(\alpha,\tau)\leq v(y_1(\alpha,\tau),\tau)-2\e_0.
\end{align}
\item  [\textbf{Case 2:}]
\begin{align}\label{eq:contradiction2}
\partial_ty_1(\alpha,\tau)\geq v(y_1(\alpha,\tau),\tau)+2\e_0.
\end{align}
\end{enumerate}
Next, we derive a contradiction from Case 1. Let $y=y_1(\alpha,\tau)$ and $a_-=v(y,\tau)-\e_0$. Then, the first inequality in \eqref{eq:claim3} holds for $t-\tau>0$ small enough. Due to \eqref{eq:contradiction}, we have
\[
y_1(\alpha,t)<y_1(\alpha,\tau)+[v(y_1(\alpha,\tau),\tau)-\e_0](t-\tau)=y+a_-(t-\tau)
\]
for $t-\tau$ small enough. Combining the first inequality in \eqref{eq:claim3} yields
\begin{equation}\label{eq:two}
\begin{aligned}
\beta_1(t)&=x_1(\beta_1(t),t)+\int_{(-\infty,x_1(\beta_1(t),t))}v_x^2(x,t)\di x=y_1(\alpha,t)+\int_{(-\infty,y_1(\alpha,t))}v_x^2(x,t)\di x\\
&<y+a_-(t-\tau)+\int_{(-\infty,y+a_-(t-\tau))}v_x^2(x,t)\di x\\
&\leq y_1(\alpha,\tau)+[v(y_1(\alpha,\tau),\tau)-\e_0](t-\tau)+\int_{(-\infty,y_1(\alpha,\tau))}v_x^2(x,\tau)\di x\\
&=\beta_1(\tau)+[v(y_1(\alpha,\tau),\tau)-\e_0](t-\tau).
\end{aligned}
\end{equation}
Therefore, 
\[
\e_0<v(y_1(\alpha,\tau),\tau)-\frac{\beta_1(t)-\beta_1(\tau)}{t-\tau}=v(x_1(\beta_1(\tau),\tau),\tau)-\frac{\beta_1(t)-\beta_1(\tau)}{t-\tau}.
\]
Passing to the limit as $t\to \tau$ in the above inequality, and using \eqref{eq:ode1}, we obtain $\e_0\leq 0$, which is a contradiction.

For Case 2, we can use the second inequality in \eqref{eq:claim3} to derive a contradiction in a similar manner.

From \eqref{eq:flowmap}, we have
$
y_1(\alpha,t)=\bar{x}(\alpha)+\int_0^tv(y_1(\alpha,s),s)\di s$, and hence, \eqref{eq:energyconserved} follows directly from \eqref{eq:combine}.

\textbf{Step 3.} Uniqueness. Notice that $v$ satisfies the Hunter-Saxton equation \eqref{eq:HS} in $L_{loc}^2(\mathbb{R}\times\mathbb{R})$, and hence, $v$ satisfies \eqref{eq:HS} classically almost everywhere in $\mathbb{R}\times\mathbb{R}$. Since we have already proved that there exists a characteristics function $y_1(\alpha,t)$ satisfying \eqref{eq:flowmap} and \eqref{eq:energyconserved}, we also have
\begin{align*}
\frac{\partial^2}{\partial t^2}y_1(\alpha,t)=\frac{\di}{\di t}v(y_1(\alpha,t),t)=(v_t+vv_x)(y_1(\alpha,t),t)=\frac{1}{2}\int_{(-\infty, y_1(\alpha,t))}v_x^2(x,t)\di x=\frac{1}{2}\left(\alpha-\bar{x}(\alpha) \right).
\end{align*}
Since
\[
\frac{\partial}{\partial t}y_1(\alpha,0)=\bar{u}(\bar{x}(\alpha)),
\]
we have
\[
y_1(\alpha,t)=\bar{x}(\alpha)+\bar{u}(\bar{x}(\alpha))t+\frac{t^2}{4}(\alpha-\bar{x}(\alpha))=y(\alpha,t),
\]
where $y(\alpha,t)$ is defined by \eqref{eq:xbarbeta}. By Definition~\eqref{eq:measuresolutionu} of $u$, we have
\begin{align}
v(y(\alpha,t),t)=\frac{\partial}{\partial t}y(\alpha,t)=\bar{u}(\bar{x}(\alpha))+\frac{t}{2}(\alpha-\bar{x}(\alpha))=u(y(\alpha,t),t).
\end{align}

\end{proof}

\noindent\textbf{Acknowledgements} Y. Gao is supported by the Start-up fund from The Hong Kong Polytechnic University. T. K. Wong is partially supported by the HKU Seed Fund for Basic Research under the project code 201702159009, the Start-up Allowance for Croucher Award Recipients, and Hong Kong General Research Fund (GRF) grant ``Solving Generic Mean Field Type Problems: Interplay between Partial Differential Equations and Stochastic Analysis'' with project number 17306420.

\appendix

\section{Some useful facts from real analysis}\label{app:real}
In this appendix, we state and prove three useful lemmas from real analysis. All of them are fundamental and somewhat classical. For readers' convenience, we also provide their proofs here.

\begin{lemma}\label{lmm:twoclaims}
The following two statements holds:
\begin{enumerate}
\item[(i)] The real line $\mathbb{R}$ cannot be written as the union of uncountably many disjoint subsets with positive measures. 
\item[(ii)] Let $X:\mathbb{R}\to\mathbb{R}$ be an absolutely continuous function satisfying $X_\xi(\xi)>0$ for a.e. $\xi\in\mathbb{R}$. Then, there exists a unique absolutely continuous inverse of $X$.
\end{enumerate}
\end{lemma}
\begin{proof}
(i) 
We prove this by a contradiction argument. Let  $\Lambda$ be an uncountable index set and   $\{A_\alpha\}_{\alpha\in \Lambda}$ be a family of uncountable disjoint subsets of $\mathbb{R}$ with $\mathcal{L}(A_\alpha)>0$ such that 
\[
\mathbb{R}=\cup_{\alpha\in\Lambda}A_\alpha.
\]
Because $\mathcal{L}(A_\alpha)>0$, there must be some  integer $n$ such that $\mathcal{L}(A_{\alpha}\cap[n,n+1])>0$. Since $\Lambda$ is an uncountable set, there must be some integer $n_0$ such that 
\[
\Lambda_0=\{\alpha\in\Lambda:~~\mathcal{L}(A_{\alpha}\cap[n_0,n_0+1])>0\}
\]
is an uncountable set. Then there exists a positive integer $m>0$ such that the set
\[
\Lambda_m=\left\{\alpha\in\Lambda:~~\mathcal{L}(A_{\alpha}\cap[n_0,n_0+1])>\frac{1}{m}\right\}
\]
is uncountable. Since the sets $\mathcal{L}(A_{\alpha}\cap[n_0,n_0+1])$  for $\alpha\in \Lambda_m$ are disjoint with each other, this contradicts with $\mathcal{L}([n_0,n_0+1])=1$.

(ii) We prove this on arbitrary interval $[a,b]$. Clearly $X$ is continuous and strictly increasing on $[a,b]$ and hence, it has a continuous and strictly increasing inverse $Z$. Let 
\[
A=\{\xi\in[a,b]:~~0<X_\xi(\xi)<\infty\},\quad B=[0,1]\setminus A,
\]
and
\[
C=X(A),\quad D=X(B).
\]
We have $\mathcal{L}(B)=0$ and $C\cup D=[X(a),X(b)]$. Since $X$ is absolutely continuous, it satisfies Luzin N property and $\mathcal{L}(X(B))=\mathcal{L}(D)=0$. 

We know that $Z$ has a finite positive derivative at each point of $C$. Let $E\subset C$ satisfying $\mathcal{L}(E)=0$. Then
\[
E=\cup_{n=1}^\infty E_n,
\]
where $E_n=\{x\in E:~~0<Z_x(x)\leq n\}$.  Then, we have $\mathcal{L}(E_n)=0$ and
\[
\mathcal{L}(Z(E))\leq \sum_{n=1}^\infty\mathcal{L}(Z(E_n))\leq \sum_{n=1}^\infty n\mathcal{L}(E_n)=0,
\]
where we have used $\mathcal{L}(E_n)\leq n\mathcal{L}(E_n)$; see \cite[Lemma 6.3]{saks1937theory}.
Hence, the inverse $Z$ on $[X(a),X(b)]$ is a continuous function of bounded variation satisfying the Luzin N property, which means that $Z$ is absolutely continuous.

\end{proof}

We have the following lemma for push-forward measures:
\begin{lemma}\label{lmm:keylemma} 
Let $X:\mathbb{R}\to\mathbb{R}$ be a continuous increasing surjective function. Define two pseudo-inverse functions of $X$ by
\[
Z_1(x)=\inf\{\xi:~~X(\xi)=x\},\quad Z_2(x)=\sup\{\xi:~~X(\xi)=x\}.
\]
Define
\[
A^{pp}=\{\xi:~~X_\xi(\xi)=0,~~Z_1(X(\xi))<Z_2(X(\xi))\},\quad A^{sc}=\{\xi:~~X_\xi(\xi)=0,~~Z_1(X(\xi))=Z_2(X(\xi))\},
\]
and
\[
B=\{\xi:~~X_\xi(\xi)>0\}.
\]
Here, $A^{pp}$ is defined in point-wise sense, and $A^{sc}$ and $B$ are defined in a.e. sense.

Let $g\ge 0 \in L^1(\mathbb{R})$ . Consider the measure 
\[
\mu =X\# (g\di \xi).
\]
Let $g_1=g\cdot 1_B$, $g_2=g\cdot 1_{A^{pp}}$ and $g_3=g\cdot 1_{A^{sc}}$, here $1_{A^{pp}}$, $1_{A^{sc}}$ and $1_B$ are characteristic functions on $A^{pp}$, $A^{sc}$ and $B$ respectively. Then the following statements hold:
\begin{enumerate}
\item[(i)]  $\mathcal{L}(X(A^{pp}\cup A^{sc}) )=0$, where $\mathcal{L}$ is the Lebesgue measure.
\item[(ii)] The absolutely continuous part of $\mu$ is given by $X\#(g_1\di \xi)$.
\item[(iii)] The set  $X(A^{pp})$ is a countable set, and the pure point part of $\mu$ is given by $X\#(g_2\di \xi)$.
\item[(iv)]  The singular continuous part of $\mu$ is given by $X\#(g_3\di \xi)$. 
\end{enumerate}  
\end{lemma}	

\begin{proof}
(i)  It follows from the change of variable formula (see \cite[Theorem 4]{serrin1969general} for instance) that
\begin{equation}\label{eq:zero}
\int_{\mathbb{R}} 1_{X(A^{pp}\cup A^{sc})}(x) \di x= \int_{\mathbb{R}} 1_{A^{pp}\cup A^{sc}} (\xi) X_{\xi}(\xi)\di \xi =0.
\end{equation}

(ii) We prove that $X\#(g_1\di \xi)\ll \mathcal{L}$. Let $E\subset \mathbb{R}$ satisfy $\mathcal{L}(E)=0$. Then
\[
[X\#(g_1\di \xi)](E)=\int_{X^{-1}(E)}g_1(\xi)\di \xi=\int_{B\cap X^{-1}(E)}g(\xi)\di \xi.
\]
Hence, we only need to show $\mathcal{L}(B\cap X^{-1}(E))=\mathcal{L}(Z_1(X(B)\cap E))=0$. This means $Z_1$ has Luzin N property on $X(B)$. Set $\tilde{B}=X(B)$.  Since 
\[
\tilde{B}\cap E=\cup_{n=1}^\infty (\tilde{B}_n\cap E),\quad\mbox{where } \tilde{B}_n:=\left\{x\in \tilde{B}:~~0<Z_{1x}(x)\leq n\right\},
\]
and
\[
\mathcal{L}(Z_1(\tilde{B}_n\cap E))\leq n\mathcal{L}(\tilde{B}_n\cap E)=0,
\]
we have 
\[
\mathcal{L}(Z_1(\tilde{B}\cap E))\leq\sum_{n=1}^\infty\mathcal{L}(Z_1(\tilde{B}_n\cap E))=0.
\] 
Hence, $X\#(g_1\di \xi)\ll \mathcal{L}$. Due to (i), we have $\mu - X\#(g_1\di \xi) \perp \mathcal{L}$, and hence, $X\#(g_1\di \xi)$ is the absolutely continuous part of $\mu$. 

(iii) Suppose $ A^{pp} \neq \emptyset$.
Let $x$, $y\in X(A^{pp})$ and $x\neq y$. By the definition of $X(A^{pp})$, we have
\[
Z_1(x)<Z_2(x),\quad Z_1(y)<Z_2(y).
\]
For each $x\in X(A^{pp})$, we choose a rational number in $[Z_1(x),Z_2(x)]$ to stand for it. Due to the increasing nature of $X$, we  have $[Z_1(x),Z_2(x)]\cap [Z_1(y),Z_2(y)]=\emptyset$, which implies different points in $X(A^{pp})$ correspond to different rational numbers. Therefore, the set $X(A^{pp})$ is at most countable.

For $\xi\in A^{pp}$, denote $x=X(\xi)$. Then, we have $z_1(x)<z_2(x)$ and
\[
\mu(\{x\})=[X\#(g\di \xi)](\{x\})=\int_{X^{-1}(\{x\})} g(\xi)\di\xi=\int_{[z_1(x),z_2(x)]}g(\xi)\di \xi \ge 0.
\]
Therefore, there is a pure point measure of $\mu$ at $x$ if the above value is strictly positive. Moreover, there is no other pure point parts as for $\xi \in A^{sc} \cup B$, $Z_1(X(\xi)) = Z_2(X(\xi)) $ and the above value must be 0.

(iv) It is direct to see from the definition that $\mu = X\#(g_1\di\xi) + X\#(g_2\di\xi) + X\#(g_3\di\xi)$.
We have proven  in (i) and (ii) that $ X\#(g_1\di\xi)$ is the the absolutely continuous part of $\mu$ and $ X\#(g_2\di\xi)$ is the pure point part of $\mu$, hence, by Lebesgue decomposition theorem,  the remaining part $ X\#(g_3\di\xi)$ is the singular continuous part of $\mu$. 
\end{proof}

The last lemma shows that a locally absolutely continuous (i.e., absolutely continuous on any closed and bounded interval of $\mathbb{R}$) function has to be globally absolutely continuous if its derivative is in $L^p(\mathbb{R})$ for some $p\in[1,\infty]$.
\begin{lemma}\label{lem:locally and global absolutely continuous}
Let $1\le p \le \infty$. Then any locally absolutely continuous function  $u$ on $\mathbb{R}$  with its derivative $u_x \in L^p(\mathbb{R})$ is globally absolutely continuous. 
\end{lemma}
\begin{proof} 
 If $p=1$, it follows from the absolute continuity of Lebesgue integral that $u$ must be globally absolutely continuous. If $p = \infty$, then $u$ is globally Lipschitz continuous, and hence, globally absolutely continuous.  

Let $1<p<\infty$. For any $\epsilon>0$, we are going to find some $\delta>0$,   such that for any disjoint open intervals $\{(a_i, b_i)\}_{i=1}^{k}$ with $\sum_{i=1}^{k}|b_i-a_i| < \delta$, we have $\sum_{i=1}^{k}|u(b_i)-u(a_i)| < \e$. For this purpose, we define $E_n: = \{x: n\le |u_x (x)|<n+1 \}$ for all integer $n\ge 0$. Let $n_0$ be a large enough integer such that $n_0^{p-1} \ge \frac{2}{\epsilon} \|u_x\|^p_{L^p}$. Then 
\begin{equation*}
\begin{aligned}
\sum_{i=1}^{k}|u(b_i)-u(a_i)| &= \sum_{i=1}^{k} \left| \int_{a_i}^{b_i} u_x(x) \di x \right|\le \sum_{i=1}^{k} \int_{a_i}^{b_i} |u_x(x)| \di x 
\\
&= \sum_{n=0}^{n_0-1}\int_{ \left(\cup_{i=1}^{k}(a_i, b_i)\right) \cap E_n}|u_x(x)| \di x +\sum_{n=n_0}^{\infty}\int_{ \left(\cup_{i=1}^{k}(a_i, b_i)\right) \cap E_n}|u_x(x)| \di x\\
&\le n_0 \delta + \sum_{n=n_0}^{\infty}\int_{ \left(\cup_{i=1}^{k}(a_i, b_i)\right) \cap E_n}|u_x(x)| \di x.\\
\end{aligned}
\end{equation*} 
Moreover,
\begin{equation*}
\begin{aligned}
\sum_{n=n_0}^{\infty}\int_{ \left(\cup_{i=1}^{k}(a_i, b_i)\right) \cap E_n}|u_x(x)| \di x
& \le  \sum_{n=n_0}^{\infty}\int_{  E_n}|u_x(x)| \di x 
\le \sum_{n=n_0}^{\infty} \frac{1}{n^{p-1}}\int_{  E_n}|u_x(x)|^p \di x \\ 
& \le  \frac{1}{n_0^{p-1}}\|u_x\|^p_{L^p} \le \frac{\epsilon}{2}.
\end{aligned}
\end{equation*} 
Now we choose $\delta < \frac{\epsilon}{2n_0}$, combining the above two inequalities  we have
\begin{equation*}
\sum_{i=1}^{k}|u(b_i)-u(a_i)| < \frac{\epsilon}{2} + \frac{\epsilon}{2} = \epsilon.
\end{equation*} 

\end{proof}

In the following example, we are going to use a fat Cantor set to create an initial datum $(\bar{u},\bar{\mu})$ such that the singular continuous part of the corresponding energy measure $\mu(t)$ is nonzero at $t=2$.  
\begin{example}[Fat Cantor set for $A_t^{L,sc}$]\label{ex:singularremark}
	
\

$\bullet$ \textbf{Fat Cantor set:}
The fat Cantor set that we use is defined as follows: we consider the set $[0,1]$, in the first step, we remove the middle open interval  with length $\frac{1}{4}$ (i.e. $[\frac{3}{8}, \frac{5}{8}]$) from   $[0, 1]$. At $n$-th step, we remove the open sub-intervals of length $\frac{1}{4^n}$ from the middle of each of the $2^{n-1}$ remaining intervals.  Continuing this procedure, the fat Cantor set $E$ is defined as the points that are never removed. We have the following well-known results:  (i) set $E$ is closed; (ii) set $E$ is nowhere dense in $[0,1]$; (iii)  $\mathcal{L}(E)=\frac{1}{2}$; and (iv) all the endpoints of the removed intervals are dense in set $E$.

$\bullet$ \textbf{A function vanishing exactly on $E$:} From 
Now, we use the above fat Cantor set to define a continuous function $h$. We choose a sequence of smooth non-negative function $h_n(x)$ such that each  $h_n(x)$ is positive inside each open interval removed at the $n$-th step and $h_n(x)=0$ outside the intervals removed at the $n$-th step. We also assume that $\sup_{x\in [0,1]} h_n(x) = 1$. Then we  define  
\[
h(x) = \sum_{n=1}^{\infty} \frac{h_n(x)}{n^2}.
\] 
As the uniform limit of a sequence of continuous functions, the function $h$ is continuous. Since $h_n$'s are supported in disjoint sets and $0\leq h_n \leq 1$ for all positive integer $n$, we also have $0\leq h \leq \frac{\pi^2}{6}$. Since $h=0$ on the endpoints of all the removed intervals, combing the property (iv) above yields $h(x)=0$ for all $x\in E$.  Moreover, $h(x)>0$ for $x \in [0,1]\backslash E$.

$\bullet$ \textbf{Initial datum $(\bar{u},\bar{\mu})$:} Finally, our target initial datum in this example is defined by
\[
\bar{u}_x(x)=\left\{
\begin{aligned}
&h^{\frac{1}{2}}(x)-1,~~x\in[0,1],\\
&0,~~\textrm{otherwise},
\end{aligned}
\right.
\]
and
\[
\bar{u}(x)=\int_{(-\infty,x)} \bar{u}_x(\eta)\di \eta,\quad \di\bar{\mu}=\bar{u}_x^2\di x.
\]
Hence, the function $\bar{u}\in C_b(\mathbb{R})$ and $\bar{u}_x\in L^2(\mathbb{R})$. Since $\bar{\mu}$ is absolutely continuous, we know that $\bar{x}(\alpha)$ (defined by \eqref{eq:barx1}) is strictly increasing and $\mathcal{L}(\mathbb{R}\setminus B_0^L)=0$ for $B_0^L$ given in Proposition \ref{pro:keypro} (iv).
From \eqref{eq:separation}, we have the derivative of characteristic $y(\alpha,t)$ at $t=2$:  
\begin{equation*}
y_{\alpha}(\alpha,2)=\bar{x}'(\alpha)\left[1+
\bar{u}_x(\bar{x}(\alpha))\right]^2=\bar{x}'(\alpha)h(\bar{x}(\alpha)),\quad \alpha\in B_0^L ~\textrm{and $\bar{x}(\alpha) \in [0,1]$}.
\end{equation*}
and
\begin{equation*}
y_{\alpha}(\alpha,2)=\bar{x}'(\alpha)\left[1+
\bar{u}_x(\bar{x}(\alpha))\right]^2=\bar{x}'(\alpha)>0,\quad \alpha\in B_0^L ~\textrm{and $\bar{x}(\alpha) \notin [0,1]$}.
\end{equation*}
It follows from Remark~\ref{rmk:thm21} (iii) that $\bar{x}'(\alpha) = \frac{t^2}{t^2+4}$ and $f(\alpha)=\frac{4}{t^2+4}$ on $A_t^{L,sc}$, so $A_2^{L,sc}=\{\alpha:~~y_\alpha(\alpha,2)=0\}=\bar{x}^{-1}(E)$ has a positive measure. Hence, the inverse of $y(\cdot,2)$ does not have Luzin N property since $\mathcal{L}(y(A_2^{L,sc},2))=0$ and $\mathcal{L}(A_2^{L,sc})=\mathcal{L}(\bar{x}^{-1}(E))=2\mathcal{L}(E)=1>0$. Moreover,
\[
\mu_{sc}(2)(\mathbb{R})=\int_{A_2^{L,sc}} f(\alpha)\di \alpha=\frac{1}{2}\mathcal{L}(\bar{x}^{-1}(E))=\mathcal{L}(E) = \frac{1}{2}.
\]
In the above construction, $\bar{u}_x(x)$ is  continuous, and $1+\bar{u}_x(x)$ vanishes on $E$.
We remark that one can have a smooth function $\bar{u}$  such that $1+\bar{u}_x(x)$ vanishes exactly at the closed set $E$ by Whitney's theorem \cite[Theorem I]{whitney1934analytic}. Obviously, the construction above for time $t=2$ can be adapted to any time $t\ne 0$.
\end{example}

\bibliographystyle{plain}
\bibliography{bibofChara}

\end{document}